\documentclass[11pt]{amsart}

 \pdfoutput=1

\usepackage{hyperref} 
\usepackage[dvips]{graphicx}
\usepackage[T1]{fontenc}
\usepackage{mathptmx}

\usepackage{color}

\DeclareMathOperator{\trace}{trace}
\DeclareMathOperator{\Ad}{Ad}

\newcommand{\LL}{\mathbb{L}}
\newcommand{\bbar}{\begin{pmatrix}}
\newcommand{\ebar}{\end{pmatrix}}

\newcommand{\G}{\mathcal{G}}

\newcommand{\gggg}{\mathcal{G}}
\newcommand{\bdm}{\begin{displaymath}}
\newcommand{\edm}{\end{displaymath}}
\newcommand{\beq}{\begin{equation}}
\newcommand{\beqa}{\begin{eqnarray}}
\newcommand{\beqas}{\begin{eqnarray*}}
\newcommand{\eeq}{\end{equation}}
\newcommand{\eeqa}{\end{eqnarray}}
\newcommand{\eeqas}{\end{eqnarray*}}
\newcommand{\dd}{\textup{d}}

\newcommand{\E}{{\mathbb E}}

\newcommand{\C}{{\mathbb C}}
\newcommand{\D}{{\mathbb D}}

\newcommand{\real}{{\mathbb R}}
\newcommand{\SSS}{{\mathbb S}}

\newcommand{\sym}{\mathcal{S}}
\newcommand{\B}{\mathcal{B}}

\newcommand{\ip}[2]{\langle#1,#2\rangle}
\newcommand{\ipb}[2]{\biggl\langle#1,#2\biggr\rangle}

\newcommand{\sign}{\mathrm{sign}}

\newcommand{\SLR}{\mathrm{SL}(2,\real)}
\newcommand{\slR}{\mathfrak{sl}(2,\real)}
\newcommand{\SLC}{\mathrm{SL}(2,\C)}
\newcommand{\slC}{\mathfrak{sl}(2,\C)}
\newcommand{\SU}{\mathrm{SU}(2)}
\newcommand{\su}{\mathfrak{su}(2)}

   \newtheorem{theorem}{Theorem}[section]
   \newtheorem{proposition}[theorem]{Proposition}
   \newtheorem{corollary}[theorem]{Corollary}
   \newtheorem{lemma}[theorem]{Lemma}
   \newtheorem{definition}[theorem]{Definition}

 \theoremstyle{remark}
   \newtheorem{example}[theorem]{Example}
   \newtheorem{remark}[theorem]{Remark}
   
\numberwithin{equation}{section}

\begin{document}

\title[The geometric Cauchy problem]{The Geometric Cauchy Problem for Surfaces With Lorentzian Harmonic Gauss maps}

\begin{abstract}
The geometric Cauchy problem for a class of surfaces in a 
 pseudo-Riemannian manifold of dimension $3$ is to find the surface which contains a given curve with a prescribed tangent bundle  along the curve.  We consider this problem for constant negative Gauss curvature surfaces (pseudospherical surfaces) in Euclidean 3-space, and for timelike  constant non-zero mean curvature (CMC) surfaces in  the Lorentz-Minkowski 3-space. We prove that there is a unique solution if the prescribed curve is non-characteristic, and for characteristic initial curves  (asymptotic curves for pseudospherical surfaces and null curves for timelike CMC) 
it is necessary and sufficient for similar data 
 to be prescribed along an additional characteristic curve that intersects the first.  
The proofs also give a means of constructing all solutions using loop group techniques.
The method used is the infinite dimensional d'Alembert type representation 
 for surfaces associated with Lorentzian harmonic maps (1-1 wave maps) into symmetric spaces, developed since the 1990's.  Explicit formulae for the potentials in terms of the prescribed data are given, and some applications are considered.
\end{abstract}

\author{David Brander}
\address{Department of Mathematics\\ Matematiktorvet, Building 303 S\\
Technical University of Denmark\\
DK-2800 Kgs. Lyngby\\ Denmark}
\email{D.Brander@mat.dtu.dk}

\author{Martin Svensson}
\address{Department of Mathematics \& Computer Science\\ and CP3-Origins, Centre of Excellence for Particle Physics Phenomenology\\
  University of Southern Denmark\\ Campusvej 55\\ DK-5230 Odense M\\
   Denmark}
\email{svensson@imada.sdu.dk}

\keywords{Differential geometry, integrable systems, loop groups, constant mean curvature surfaces, pseudospherical surfaces}

\subjclass[2000]{Primary 53A05, 53A35; Secondary 53A10, 53C42, 53C43}

\thanks{Research partially sponsored by CP3-Origins DNRF 
Centre of Excellence Particle Physics Phenomenology. Report no. 
CP3-ORIGINS-2010-40.}

\maketitle


\section{Introduction}

The \emph{geometric Cauchy problem} for a class of surfaces immersed in a manifold $N$  is to find all surfaces
of this class which contain some specified curve and with the surface tangent bundle prescribed along this curve.  This is nothing other than the classical  Bj\"orling 
problem for minimal surfaces addressed to other surface classes.

For minimal surfaces there is a unique solution given by a simple formula, because 
these surfaces have the Weierstrass representation in terms of holomorphic functions,
and the prescribed data is sufficient to determine these holomorphic functions along 
the curve. The solution is then given by analytic extension.
 The geometric Cauchy problem has recently  been studied in several  
situations involving holomorphic representations of surface classes
-- see, for example,  \cite{gm2005, gm2005-2, galvezetal2007, dhkw, galvezmira2004, mira2006, aliaschavesmira2003, bjorling, kimyang2007, sbjorling}.  The solution of this problem
is clearly a useful tool, both for proving general local properties of the surfaces and 
for constructing interesting examples.

The associated partial differential equations (PDE) in the works referred to above are elliptic, and the solutions therefore real analytic.  It is interesting to see what can be done with  
hyperbolic equations.   Aledo, G\'alvez and Mira have recently shown
that the geometric Cauchy problem can be solved for flat surfaces in the 3-sphere \cite{algami2009}, which are associated to the homogeneous wave equation, a hyperbolic
problem.  The situation is quite different from the elliptic case, since the solutions
are not real analytic in general. Nevertheless,  for non-characteristic data,  a unique solution is given in \cite{algami2009} using a d'Alembert type construction.

In this article we aim to address the geometric Cauchy problem for surfaces associated to  harmonic maps from a Lorentzian surface into a Riemannian symmetric space.  The associated PDE for the specific surfaces we will discuss are the 
sine-Gordon equation, the hyperbolic cosh/sinh-Gordon equations and the 
Liouville equation.  There are no classical d'Alembert type solutions to these; 
however Lorentzian harmonic maps have a
 loop group representation \cite{melkosterling, terng1997}. From loop group techniques a kind of infinite dimensional d'Alembert solution can be found, whereby all solutions are given in terms of two functions, each of one variable only \cite{todaagag, dit2000}.
This type of solution was also derived in the late 1970's by Krichever for the sine-Gordon equation \cite{krichever}. 
 
Examples of surfaces associated to Lorentzian harmonic maps
 include constant  Gauss curvature $-1$ surfaces in $\E^3$ (pseudospherical surfaces),  
timelike constant mean curvature (CMC)  surfaces in 
Lorentz-Minkowski 3-space $\LL^3$, and spacelike constant positive Gauss curvature surfaces in $\LL^3$.  Specifically, we treat  the first two of these cases; 
since the main tool is the d'Alembert type solution for Lorentz harmonic maps,
 one expects that the approach can be adapted to other such problems.
 
We  treat both pseudospherical surfaces and timelike CMC surfaces because the problem is not
identical for the two cases.  Firstly, the Gauss map of a timelike CMC surface is Lorentzian
harmonic with respect to the first fundamental form of the surface, while for
pseudospherical surfaces the Gauss map is harmonic with respect to the Lorentzian  metric 
given by the second fundamental form.  Consequently the routes to solving the problems are slightly different.  Secondly, the group involved in the 
construction is compact in the pseudospherical case and non-compact in the timelike CMC case.  
The non-compact case is interesting because the loop group decomposition used
is not global, an issue which will be discussed in future work.

\subsection{Results of this article}  \label{resultssection}
By \emph{uniqueness} of the solution of the geometric Cauchy problem,  we will always mean
the following:  given two solutions $f: M \to N$ and $\tilde f: \tilde M \to N$, then,  at any point $p=\gamma(t_0)$ on
the initial curve $\gamma$,  with $f(z_0)=\tilde f(\tilde z_0)=p$, there are
neighbourhoods $U$ of $z_0$ and $\tilde U$ of $\tilde z_0$, and
an isometry $\phi: U \to \tilde U$ such that $f= \tilde f \circ \phi$.

In Section \ref{tcmcgcpsection} we solve the geometric Cauchy problem for timelike
CMC surfaces. In Theorem \ref{gcptheorem} we prove existence and uniqueness of the solution for timelike or spacelike initial curves, and in Theorem \ref{bptheorem} we give a simple formula for
the potentials used to construct the solution in terms of the initial data.
If the initial curve is a null curve, we prove in 
Theorem \ref{th:GCPnull} that it is necessary and sufficient to specify 
similar geometric Cauchy data on an additional null curve, which intersects
the first, to obtain a unique solution. The potentials are also given explicitly.

As an application, we use in Section \ref{revsection} the solution of the geometric Cauchy problem to find the potentials for timelike CMC surfaces of revolution.

In Section \ref{ksurfsection} we solve the geometric Cauchy problem for 
pseudospherical surfaces.  Given an initial curve $f_0$ and prescribed 
surface normal $N_0$ along the curve, the characteristic case
is distinguished this time not by the curve being null, but by the vanishing of the
inner product $\ip{f_0^\prime}{N_0^\prime}$, which dictates that the curve must
be an asymptotic curve of any solution surface.  

Theorem \ref{ksurftheorem1} states that, if $\ip{f_0^\prime}{N_0^\prime}$ does
not vanish, then there is a unique regular solution to the problem, provided 
that, in addition, $f_0^\prime$ and $N_0^\prime$ are either everywhere parallel
or nowhere parallel. Explicit formulae for the potentials are given.
The characteristic case, where $\ip{f_0^\prime}{N_0^\prime}$ is everywhere zero is treated in
Section \ref{asymptsection}, and here one needs to specify an additional complex
function to obtain a unique solution.

As a consequence of Theorem \ref{ksurftheorem1}, 
 Corollary \ref{ksurfcor} states that, given a space curve
with non-vanishing curvature, with the additional condition 
that the torsion is either zero everywhere or never zero, then there is a unique pseudospherical surface which contains 
this curve as a geodesic.  This geodesic is, of course, a principal curve if and only if
the curve is a plane curve.  Some examples are computed numerically 
(Figures \ref{ellipsefigure},  \ref{parabolafigure}, \ref{cubicfigure}
and \ref{lemfigure}).

\begin{figure}[here] 
\begin{center}
\includegraphics[height=45mm]{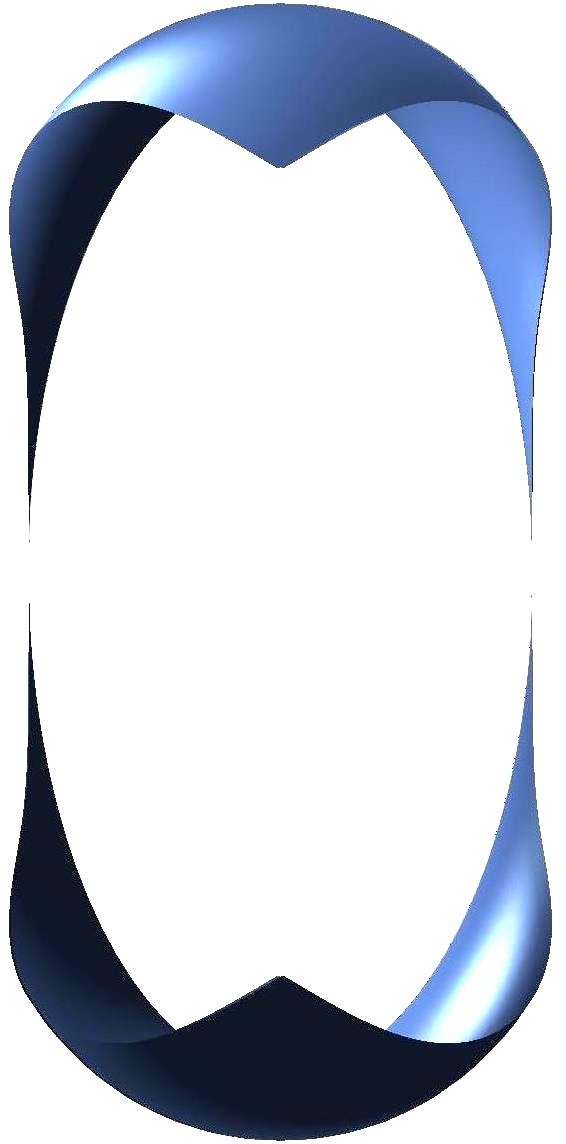} \hspace{2cm}
\includegraphics[height=45mm]{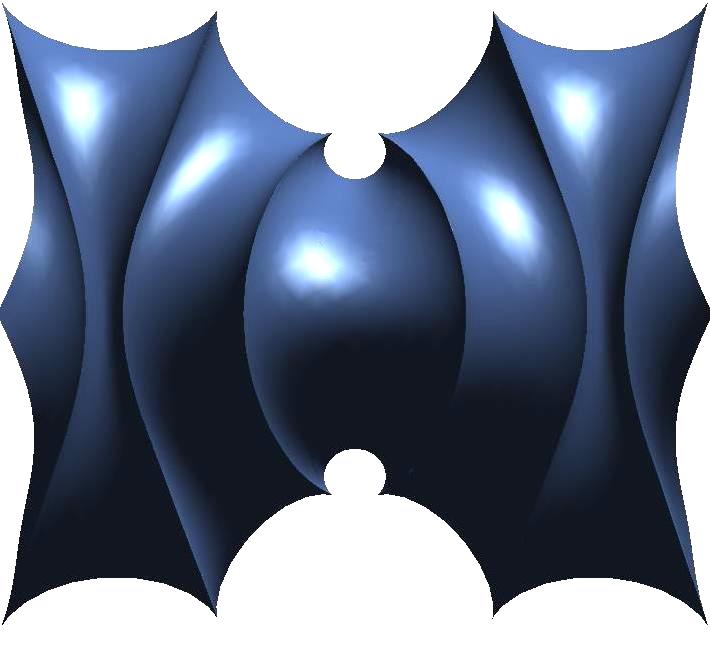}
\end{center}
\caption{Two views of the unique pseudospherical surface that contains the ellipse
 $x^2+(y/2)^2=1$ as a geodesic principal curve. The ellipse  wraps around the 
 smooth central band of the image on the right.}  \label{ellipsefigure}
\end{figure}

\begin{remark}
Generally, we do not discuss the PDE associated to the geometric problems here. The non-characteristic Cauchy problem for the sine-Gordon equation 
has been well studied  within
the class of rapidly decreasing functions; with that type of initial data,
it was solved by Ablowitz et al \cite{akns1}
using inverse scattering.
We point out here that the construction in Section \ref{ksurfsolsection} can be used to prove global existence and uniqueness for the solution to the sine-Gordon equation with arbitrary $C^\infty$
Cauchy data along a non-characteristic curve, and the solutions can be computed numerically by solving an  ordinary differential equation and performing an LU decomposition. The global existence of the solution follows from the global Birkhoff decomposition proved in \cite{jgp}. For the PDE associated to the timelike CMC surface the solution is unique for non-characteristic data, but only proved here to exist on a (large) open set containing the initial curve.
\end{remark}

\noindent \textbf{Notation:} Throughout this article we follow the convention that, if $\hat X$ denotes any object which 
depends on the loop parameter $\lambda$, then dropping the hat means we evaluate at 
$\lambda = 1$, that is $X := \hat X \big|_{\lambda=1}$.


\section{The infinite dimensional d'Alembert solution for timelike CMC surfaces}
We first summarize the method given by Dorfmeister, Inoguchi and Toda \cite{dit2000} for constructing all timelike CMC surfaces from pairs of functions of one variable (called potentials) via a loop group splitting.  The essential idea is the same as
that used earlier by Toda \cite{todaagag} for pseudospherical surfaces in $\E^3$.
The idea of using the loop parameter and the Sym formula to obtain the surface
can be traced back to Sym \cite{sym1985}.  

\subsection{The $\SLR$ frame for a Lorentz conformal immersion}\label{extendedframesection}

We denote by $\LL^3$ the Lorentz-Minkowski space with metric of signature $(-++)$.Take the following basis for the Lie algebra $\slR$:
\begin{eqnarray*}  
e_0 = \bbar 0 & -1 \\1 & 0\ebar, & e_1 = \bbar 0 & 1 \\1 & 0\ebar, & 
    e_2 = \bbar -1 & 0 \\ 0 & 1\ebar.
\end{eqnarray*}
With the inner product $\ip{X}{Y}=\frac{1}{2}\trace(XY)$, the vectors $e_0,e_1,e_2$ form an orthonormal basis with $\ip{e_0}{e_0}=-1$, and we use this to identify $\LL^3 \cong \slR$. 

Let $M$ be a simply connected $2$-manifold. 
\begin{definition} A conformal immersion $f:M\to\LL^3$ is said to be \emph{timelike} if the metric on $M$ induced by $f$ has signature $(-,+)$. 
\end{definition} 

In the following, a timelike immersion will always be understood as a conformal timelike immersion. 

Let $f:M\to\LL^3$ be a timelike immersion. The metric induced by $f$  determines a Lorentz conformal structure on $M$. For any null coordinate system $(x,y)$ on $M$, we define a function $\omega: M \to \real$ by the condition that the induced metric is given by
\begin{equation}\label{first}
\dd s^2 = \varepsilon  e^\omega \, \dd x \, \dd y, \hspace{1cm} \varepsilon= \pm 1.
\end{equation}

Let $N$ be a unit normal field for the immersion $f$,  and define a \emph{coordinate frame} for $f$ to be a map $F:M\to\SLR$ which satisfies 
\begin{eqnarray}  \label{coordframe}
\begin{split}
f_x = \dfrac{\varepsilon_1}{2} e^{\omega/2} \Ad_F (e_0 + e_1), \\
f_y = \dfrac{\varepsilon_2}{2} e^{\omega/2} \Ad_F (-e_0 + e_1), \\
N = \Ad_F(e_2), 
\end{split}
\end{eqnarray}
where $\varepsilon_1,\varepsilon_2\in\{-1,1\}$. In this case \eqref{first} holds with $\varepsilon=\varepsilon_1\varepsilon_2$. Conversely, since $M$ is simply connected, we can always construct a coordinate frame for a timelike conformal immersion $f$. For a regular surface we can choose coordinates and a coordinate frame such that $\varepsilon_1=\varepsilon_2=\varepsilon=1$, but we prefer a set-up that
 can also be used for surfaces which fail to be regular at some points.

The Maurer-Cartan form $\alpha$ for the frame $F$ is defined by
$$
\alpha = F^{-1} \dd F = U \dd x + V \dd y.
$$ 
With the choices made above, one easily computes
\begin{eqnarray*}
U =  \frac{1}{4} \bbar - \omega_x & -4 \varepsilon_1 Q e^{-\frac{\omega}{2}} \\ 
       2 \varepsilon_1 H e^{\frac{\omega}{2}} &  \omega_x \ebar,\qquad 
V = \frac{1}{4} \bbar  \omega_y & -2 \varepsilon_2 H e^{\frac{\omega}{2}} \\ 
       4 \varepsilon_2 R e^{-\frac{\omega}{2}} & - \omega_y \ebar,
\end{eqnarray*}
where $H= 2\varepsilon e^{-\omega}\ip{f_{xy}}{N}$ is the mean curvature and $Q=\ip{f_{xx}}{N}$ and $R=\ip{f_{yy}}{N}$. The quadratic differentials $Q \dd x^2$ and $R \dd y^2$ are independent of the choice of null coordinates (provided the null directions are not
interchanged), and are called the \emph{Hopf differentials} of $M$.

\subsection{The extended coordinate frame for a timelike CMC surface}
The integrability condition $\dd \alpha + \alpha \wedge \alpha =0$, also known as the Maurer-Cartan equation for the $\slR$-valued 1-form $\alpha$, is equivalent to the Gauss-Codazzi equations for the surface:
\begin{gather}
\varepsilon\omega_{xy} + \frac{1}{2}H^2 e^\omega - 2 Q R e^{-\omega} = 0, \label{compat1}\\
H_x = 2 \varepsilon e^{-\omega} Q_y,\qquad H_y = 2 \varepsilon e^{-\omega} R_x. \label{compat2}
\end{gather}
From these it is clear that the mean curvature $H$ is constant if and only if $Q$ and $R$ depend only on $x$ and $y$ respectively, that is, 
\begin{eqnarray*}
H\ \textup{constant} & \Longleftrightarrow & Q_y=R_x=0.
\end{eqnarray*}
It follows easily that these conditions on $R$ and $Q$ do not depend on the choice of Lorentz null coordinates.

For any $\lambda\in\real\setminus\{0\}$, set $Q_\lambda=\lambda Q$ and $R_\lambda=\lambda^{-1} R$. Replacing $Q$ and $R$ by these in the compatibility conditions \eqref{compat1} and \eqref{compat2},  equation \eqref{compat1} is unaffected, whilst equations \eqref{compat2} are satisfied for all $\lambda$ if and only if the mean curvature $H$ is constant.  If this is the case, then the compatibility conditions are also unchanged if we multiply the $(2,1)$ and $(1,2)$ components of $U$ and $V$ respectively by $\lambda$
and $\lambda^{-1}$. Set $\hat \alpha=\hat U \dd x + \hat V \dd y$, with
\beqa  
\begin{split}
\label{withlambda}
\hat U =  \frac{1}{4} \bbar - \omega_x & -4 \varepsilon_1 Q e^{-\frac{\omega}{2}} \lambda  \\ 
       2 \varepsilon_1 H e^{\frac{\omega}{2}} \lambda &  \omega_x \ebar, \\
\hat V = \frac{1}{4} \bbar  \omega_y & -2 \varepsilon_2 H e^{\frac{\omega}{2}} \lambda^{-1} \\ 
       4 \varepsilon_2 R e^{-\frac{\omega}{2}} \lambda^{-1} & - \omega_y \ebar. 
\end{split}
\eeqa
Then we have demonstrated the following:

\begin{lemma} \label{lambdalemma}
The mean curvature $H$ is constant if and only if $\dd \hat \alpha + \hat \alpha \wedge \hat \alpha =0$ for all $\lambda \in \real\setminus\{0\}$.
\end{lemma}
By a \emph{timelike CMC} surface  we mean a timelike conformal immersion $f:M\to\LL^3$ with constant mean curvature; if $H$ is the constant value of the mean curvature we call it a timelike CMC $H$ surface.

Let $\Lambda\SLC$ denote the group of (smooth) loops with values in $\SLC$, with a suitable
topology.  The $H^s$ topology with $s>1/2$ is sufficient for our purposes.  Define the involution $\rho$ on $\Lambda\SLC$ by 
$$
(\rho \gamma) (\lambda) = \overline{\gamma(\bar \lambda)}.
$$
Denoting by $\Lambda\SLC_\rho$ the fixed point subgroup, we note that any loop $\gamma \in \Lambda\SLC_\rho$ which extends holomorphically to some neighbourhood of $\SSS^1$, is $\SLR$-valued for real values of $\lambda$.
 
Consider another involution $\sigma$ on $\Lambda\SLC$ given by 
$$
(\sigma \gamma) (\lambda) = \Ad_{e_2} \gamma(-\lambda).
$$
It is easy to see that $\sigma$ and $\rho$ commute, and we denote by $\Lambda\SLC_{\sigma}$ the subgroup of $\Lambda\SLC$ of loops fixed by $\sigma$ and by $\Lambda\SLC_{\sigma\rho}$ the subgroup fixed by both $\sigma$ and $\rho$. To indicate that $\Lambda\SLC_{\sigma\rho}$ is a real form of $\Lambda\SLC_\sigma$ we will from now on use the shorthand notation 
$$
\G^\C=\Lambda\SLC_\sigma,\quad \G=\Lambda\SLC_{\sigma\rho}.
$$

We use the same symbols to denote the infinitesimal versions of  $\rho$ and $\sigma$ on the Lie algebra $\Lambda\slC$.  Thus the Lie algebra of $\G$ is
$Lie(\G)=\Lambda\slC_{\sigma\rho}$, the subalgebra of fixed points of $\sigma$ and $\rho$ in $\Lambda\slC$,  whilst $Lie(\G^\C)=\Lambda\slC_\sigma$ is the Lie algebra of $\G^\C$. Elements of $\mathcal{G}$, $\mathcal{G}^\C$ and their Lie algebras all have the property that diagonal and off-diagonal components are even and odd functions of $\lambda$ respectively, due to the twisting introduced by $\sigma$.

Let $f:M\to\LL^3$ be a timelike CMC surface, $F$ a coordinate frame for $f$, and $(x_0,y_0)\in M$ a given point. As a consequence of Lemma \ref{lambdalemma}, 
we can integrate the equation $\hat\alpha=\hat F^{-1}\dd\hat F$ with the initial condition $\hat F(x_0,y_0)=F(x_0,y_0)$ to obtain a map $\hat F:M\to\G$ which we call an \emph{extended coordinate frame} for $f$. Note that $\hat F\big\vert_{\lambda=1}=F$. 

From the extended coordinate frame, we can easily reconstruct our surface $f$ from the so-called \emph{Sym formula}. 
Define a map $\sym: \Lambda \SLC \to\Lambda\slC$ by the formula, for
 any $\hat G \in \Lambda \SLC$,
$$
\sym(\hat G)=2\lambda\partial_\lambda\hat G\hat G^{-1}-\Ad_{\hat G}(e_2). 
$$
For any $\lambda_0$ in the Riemann sphere  $\hat \C$ at which the holomorphic extension of the
loop $\hat G$ is defined, we define $\sym_{\lambda_0}(\hat G)=\sym(\hat G)\big\vert_{\lambda=\lambda_0}$. 

\begin{lemma}  \label{symlemma}
Let $H \in \real\setminus\{0\}$ and let $f$ be a timelike CMC $H$ surface with extended frame 
$\hat F$ as described above, with $f(p) = 0$ for some point
$p \in M$. Then $f$ is recovered by the \emph{Sym formula}
$$
f (z) = \frac{1}{2H}\left \{ \sym_1 (\hat F(z)) - \sym_1(\hat F(p))\right \}.
$$
For other values of $\lambda \in \real\setminus\{0\}$, 
$f^\lambda = \frac{1}{2H} \sym_\lambda (\hat F): M \to \LL^3$
is also a timelike CMC $H$ surface, with the same null coordinate system
and metric, but with Hopf differentials 
$(Q_\lambda, R_\lambda) = (\lambda Q, \lambda^{-1} R)$.
\end{lemma}

\begin{proof}
To verify the formula for $f(z)$, set 
$\tilde f(z) = \frac{1}{2H}(\sym_1 (\hat F (z)) - \sym_1(\hat F(p)))$. It is readily verified that $\tilde f_x = f_x$ and $\tilde f_y = f_y$. Since $\tilde f(p) = f(p) = 0$, we see that $\tilde f(z) = f(z)$ for all $z \in M$.  

To verify that $f^\lambda$ is timelike CMC $H$ for other values of $\lambda$, we 
note  $f^\lambda_{x} = \lambda f_x$ and $f^\lambda_{y} = \lambda^{-1} f_y$,
so that $N$ is also the normal to this surface. It follows now easily that $H = 2\varepsilon e^{-\omega} \ip{f^\lambda_{xy}}{N}$.
\end{proof}

\subsection{The loop group characterization of timelike CMC surfaces}
The next proposition identifies the essential properties of the matrices $\hat U$ and $\hat V$ that we need in \eqref{withlambda} in order to characterize timelike CMC surfaces as loop group maps.

\begin{definition}  \label{admissibleframedef}
Let $M$ be a simply connected subset of $\real^2$, and let $(x,y)$ denote the standard 
coordinates. An \emph{admissible frame} on $M$ is a smooth map $\hat F: M \to \G$  such that the Maurer-Cartan form of $\hat F$ has the form
\begin{equation}  \label{admissible}
\hat F^{-1} \dd \hat F =\lambda\, A_{1}\, \dd x + \alpha_0 + \lambda^{-1}A_{-1}\, \dd y,
\end{equation}
where the $\slR$-valued 1-form $\alpha_0$ is constant in $\lambda$.  The admissible frame
 $\hat F$ is said to be \emph{regular} if $[A_1]_{21}\neq 0$ and $[A_{-1}]_{12}\neq 0$. 
\end{definition}
Due to the loop group twisting,  a regular admissible frame can be written
 $\hat F^{-1} \dd \hat F = \hat U \dd x + \hat V \dd y$, with
  $$
 \hat U = \bbar a_1 & b_1 \lambda \\ c_1 \lambda & -a_1 \ebar\quad \text{   and   }\quad
 \hat V = \bbar a_2 & b_2 \lambda^{-1} \\ c_2 \lambda^{-1} & -a_2 \ebar,
 $$
 where $c_1$ and $b_2$ are non-zero, and we use this notation in the next proposition.

\begin{proposition}  \label{characterizationprop}
Let $\hat F: M \to \G$ be a regular admissible frame and $H\neq0$. 
Set $\varepsilon_1=\sign(c_1)$, $\varepsilon_2=-\sign(b_2)$ and $\varepsilon=\varepsilon_1\varepsilon_2$. Define a Lorentz metric on $M$ by
$$
\dd s^2 = \varepsilon e^\omega \dd x \, \dd y, \quad
\varepsilon e^\omega = -\dfrac{4 c_1 b_2}{H^2}.
$$
Set
$$
f^\lambda =\frac{1}{2H}\sym_{\lambda}(\hat F):M\to\LL^3\qquad(\lambda\in\real\setminus\{0\}).
$$
Then, with respect to the choice of unit normal $N^\lambda = \Ad_{\hat F} e_2$ and the given metric, the surface $f^\lambda$ is a timelike CMC $H$ surface. Set
$$
\rho=\left|\frac{b_2}{c_1}\right|^{\frac{1}{4}},\qquad T=\bbar \rho & 0 \\ 0 & \rho^{-1} \ebar,  
$$
and set $\hat F_C = \hat F T : M \to \G$. Then $\hat F_C$ is the 
extended frame for the surface $f=f^1$, with respect to the coordinate frame defined by
\beqas
f_x = \varepsilon_1 \frac{1}{2} e^{\omega/2} \Ad_{F_C} (e_0 + e_1), \quad
f_y =  \varepsilon_2 \frac{1}{2}e^{\omega/2} \Ad_{F_C} (-e_0+e_1),  \\
N = \Ad_{F_C} e_2 = \Ad_F e_2.
\eeqas
\end{proposition}

\begin{proof}
Since $T$ is diagonal and constant in $\lambda$, we have $\sym( \hat F_C) = \sym (\hat F)$
and $\hat F_C^{-1} \dd \hat F_C = \hat U_C \dd x + \hat V_C \dd y$, where
\beqas
\hat U _C = \bbar \rho^{-1}\rho_x+a_1 & \rho^{-2} b_1 \lambda \\ \rho^2 c_1 \lambda 
   & -\rho^{-1}\rho_x-a_1  \ebar, \\
 \hat V_C = \bbar \rho^{-1}\rho_y+a_2 & \rho^{-2} b_2 \lambda^{-1} \\
    \rho^2 c_2 \lambda^{-1} & -\rho^{-1}\rho_y-a_2 \ebar.
\eeqas
Differentiating  $f=\frac{1}{2H}\sym_{1}(\hat F_C)$,
we use $\hat U_C$ and $\hat V_C$ to compute 
\begin{eqnarray*}
f _x =  \frac{c_1 \rho^2}{H} \Ad_{F_C}(e_0+e_1),\qquad 
f _y = - \frac{b_2 \rho^{-2}}{H} \Ad_{F_C}(-e_0+e_1).
\end{eqnarray*}
It follows that $f$ is conformally immersed with conformal factor $\varepsilon e^\omega=-\frac{4 c_1b_2}{H^2}$, and $F_C$ is the coordinate frame given at
equation (\ref{coordframe}), with $\varepsilon_i$ as defined here. Therefore $\hat F_C$ has precisely the form of an extended coordinate frame for a timelike CMC $H$ surface and, by Lemma \ref{lambdalemma}, the result follows.
\end{proof}

\subsection{The d'Alembert construction of timelike CMC surfaces}\label{dpwsection}

The point of Lemma \ref{symlemma} and Proposition \ref{characterizationprop} is
 that the problem of finding a conformal timelike CMC $H \neq 0$ immersion $M \to \LL^3$ is equivalent to finding an admissible frame.  In this section, we explain how to construct an admissible frame from two pairs of real functions.

Let $\Lambda^\pm\SLC_\sigma$ denote the subgroup of $\G^\C$ consisting of loops which extend holomorphically to $\D^\pm$, where $\D^+$ is the unit disc and $\D^- = \SSS^2 \setminus \{ \D^+ \cup \SSS^1 \}$, the exterior disc in the Riemann sphere. Set $0^+:=0$ and $0^-:=\infty$ and define 
$$
\G^{\pm}=\G\cap\Lambda^\pm\SLC_\sigma,\quad \G^\pm_*=\{\gamma\in\G^\pm\ |\ \gamma(0^\pm)=I\}.
$$
We define the complex versions $\G^{\C\pm}$ analogously by substituting $\G^\C$ for $\G$ in the above definitions. 

The Birkhoff decomposition for $\Lambda\SLC$ \cite{PreS} restricts to a decomposition of either of the subgroups $\G^\C$ or $\G$; see \cite{branderdorf}, Proposition 1, for a general statement for fixed-point subgroups with respect to finite order automorphisms of the type used here.

\begin{theorem}[The Birkhoff decomposition]\label{birkhoff}$\phantom{}$
\begin{enumerate}
\item
The sets $\B_L=\G^-\cdot\G^+$ and $\B_R=\G^+\cdot\G^-$ are both open and dense in $\G$.  The multiplication maps 
$$
\G^-_*\times\G^+\to\B_L \ \text{  and  }\ \G^+_*\times\G^-\to\B_R
$$
are both real analytic diffeomorphisms.
\item
The analogue holds substituting $\G^\C$, $\G^{\C \pm}$ and $\G^{\C \pm}_*$ for $\G$, $\G^\pm$ and $\G^\pm_*$, respectively, writing $\B_L^\C=\G^{\C-}\cdot\G^{\C+}$ and $B_R^\C=\G^{\C+}\cdot\G^{\C-}$. 
 \end{enumerate}
\end{theorem}

In particular, any element $\hat F$ in the \emph{big cell} $\B=\B_L\cap\B_R$ has exactly two decompositions 
\begin{equation} \label{birkhoffsplit}
\hat F = \hat F_- \hat H_+ = \hat F_+ \hat H_-,\ \text{ where }\
  \hat F_\pm \in \G^\pm_*,\ \text{ and }\   \hat H_\pm \in \G^\pm,
\end{equation}
and the factors $\hat F_\pm$ and $\hat H_\pm$ depend real analytically on $\hat F$.

\begin{definition} \label{potentialdefn}
Let $I_x$ and $I_y$ be two real intervals, with coordinates $x$ and $y$ respectively. A \emph{potential pair} $(\chi, \psi)$ is a pair of smooth 
$Lie(\G)$-valued 1-forms on $I_x$ and $I_y$ respectively with Fourier expansions in $\lambda$ as follows:
\begin{eqnarray*}
\chi = \sum_{j=-\infty}^1 \chi_i \lambda^i \dd x,\qquad \psi = \sum_{j=-1}^\infty \psi_i \lambda^i \dd y.
\end{eqnarray*}
The potential pair is called \emph{regular} if $[\chi_1]_{21}\neq 0$ and $[\psi_{-1}]_{12}\neq0$.
\end{definition}

The following theorem is a straightforward consequence of Theorem \ref{birkhoff}. 
\begin{theorem} \label{dpwthm}
\begin{enumerate}
\item
Let $M$ be a simply connected subset of $\real^2$ and $\hat F: M \to \B \subset \G$ an admissible frame.  The pointwise (on $M$) Birkhoff decomposition \eqref{birkhoffsplit} of $\hat F$ results in a potential pair $(\hat F_+^{-1} \dd \hat F_+ \, , \, \hat F_-^{-1} \dd \hat F_-)$,
of the form
$$
\hat F_+^{-1} \dd \hat F_+ = \chi_1 \lambda \, \dd x,\qquad
\hat F_-^{-1} \dd \hat F_- = \psi_{-1} \lambda^{-1} \, \dd y.
$$
\item
Conversely, given any potential pair, $(\chi, \psi)$, define $\hat F_+: I_x \to \G$ and $\hat F_-: I_y \to \G$ by  integrating the differential equations
\begin{eqnarray*}
\hat F_+^{-1} \dd \hat F_+ = \chi, & \hat F_+(x_0) = I,\\
\hat F_-^{-1} \dd \hat F_- = \psi, & \hat F_-(y_0) = I.
\end{eqnarray*}
Define $\hat \Phi=\hat F_+^{-1} \hat F_- : I_x \times I_y \to \G$, and set 
$M = \hat \Phi^{-1}(\B_L)$. Pointwise on $M$, perform the Birkhoff decomposition $\hat \Phi = \hat H_- \hat H_+$, where $\hat H_-: M \to \gggg^-_*$ and $\hat H_+ : M \to \G^+$.
Then $\hat F = \hat F_- \hat H_+^{-1}$ is an admissible frame.
\item In both items (1) and (2), the admissible frame is
regular if and only if the corresponding potential pair is regular. Moreover,
with notation as in Definitions \ref{admissibleframedef} and  \ref{potentialdefn},
we have $\textup{sign}[A_1]_{21} = \textup{sign}[\chi_1]_{21}$ and
$\textup{sign}[A_{-1}]_{12} = \textup{sign}[\psi_{-1}]_{12}$. In fact, we have 
$$
\hat F^{-1}\dd \hat F=\lambda\chi_1\dd x+\alpha_0+\lambda^{-1}\hat H_+\big|_{\lambda=0}\psi_{-1}\hat H_+^{-1}\big|_{\lambda=0}\dd y,
$$
where $\alpha_0$ is constant in $\lambda$. 
\end{enumerate}
\end{theorem}

Note that $\hat F$ in item (3) above is not in general an extended coordinate frame for the timelike CMC surface constructed from $\hat F$ via the Sym formula. With notation as in Proposition \ref{characterizationprop}, $\hat F_C=\hat FT$ is an extended coordinate frame for this surface, and 
$$
\hat F_C^{-1}\dd\hat F_C=\lambda T^{-1}\chi_1T\dd x+\alpha_0+T^{-1}\dd T+\lambda^{-1}T^{-1}\hat H_+(0)\psi_{-1}\hat H_+(0)^{-1}T\dd y.
$$

\section{The geometric Cauchy problem}  \label{tcmcgcpsection}

In this section we show how to construct a unique timelike conformally immersed CMC surface from appropriate data along a curve.  We treat two cases; first the case where the curve is
non-characteristic for the PDE (here meaning not a null curve), where there is a unique solution
of the geometric Cauchy problem, and next the case of a null curve, where one instead 
needs two curves to obtain a unique solution.   We do not consider curves of ``mixed type''.

\subsection{The case of non-characteristic curves}

Given a timelike immersion $f:M\to\LL^3$ and a local null coordinate system $(x,y)$ on $M$, the associated \emph{Lorentz isothermal coordinate system}
is defined by 
$$
u = \frac{1}{2}\left(x-y\right), \hspace{1cm} v= \frac{1}{2}\left(x+y\right).
$$
Note that, even though the null directions are well-defined, the directions determined by $\partial_u$ and $\partial_v$ are not, as they depend on the choice of null coordinates. With respect to the associated isothermal coordinates, a conformal metric $\dd s^2 = e^\omega \dd x \, \dd y$ is  of the form 
$$
\dd s^2 = e^\omega(-\dd u^2 + \dd v^2).
$$

The first problem which we shall solve is one where the following data is given:\\
\noindent\textbf{Geometric Cauchy data:} An open interval $J \subset \real$ containing $0$,  a regular smooth curve $f_0: J \to \LL^3$, which is either timelike or spacelike, and a regular smooth vector field $V: J \to \LL^3$ along $f_0$, with the condition that the vector fields $\frac{\dd f_0}{\dd t}(t)$ and $V(t)$ are everywhere orthogonal and
$$
\ip{V(t)}{V(t)} =-\biggl\langle\frac{\dd f_0}{\dd t}(t),\frac{\dd f_0}{\dd t}(t)\biggr\rangle\qquad(t\in J).
$$
Note that our assumptions on $V$ are equivalent to prescribing a family of timelike tangent planes along the curve $f_0(t)$.\\
\noindent \textbf{Non-characteristic geometric Cauchy problem:} Find a timelike CMC $H$-surface which contains the curve $f_0$ and is tangent along this curve to the plane spanned by $\dd f_0/\dd t$ and $V$.\\

\subsection{Existence and uniqueness} 

We state and prove the existence and uniqueness of the solution for the case that $f_0$ is timelike.  A similar result, with the obvious changes in statement and proof, holds for the case that $f_0$ is spacelike.

\begin{theorem} \label{gcptheorem}
Assume that geometric Cauchy data as above are given, with $f_0$ timelike, and let $H\neq 0$. 
Set $J^- = \{y \in \real ~|~ - y \in J \}$. Let $M = J \times J^-$, with coordinates $(x,y)$, and set 
$$
\Delta= \{(x,-x)\ |\ x\in J\} \subset J \times J^-. 
$$ 
Then:
\begin{enumerate}
\item
There is an open  subset $W \subset M$, which contains $\Delta$, and a unique timelike CMC $H$ immersion $f: W \to \LL^3$, with Lorentz isothermal coordinates $(u,v) = (\frac{1}{2}(x-y), \frac{1}{2}(x+y))$, satisfying $$
f(u,0) = f_0(u)\qquad\text{and}\qquad\frac{\partial f}{\partial v}(u,0) = V(u).
$$
\item  The surface so constructed is the unique, in the sense defined in Section \ref{resultssection}, solution to the non-characteristic geometric Cauchy problem with the orientation
given by $\frac{\dd f_0}{\dd t} \wedge V$.
\end{enumerate}
\end{theorem}

\begin{proof} 
\textbf{Item 1:}
We first show that there is a local solution in a neighbourhood of any point in $\Delta$, and then that any two solutions agree at points where they are both defined.

\textbf{Local existence:} Fix a point $t_0\in J$. Without loss of generality, we may assume that $f^\prime(t_0)$ is a multiple of $e_0$ and $V(t_0)$ a multiple of $e_2$.
We seek a solution such that $t$ corresponds to the coordinate $u$ along $v=0$.
Define the function $\omega_0:J\to\real$ by $e^{\omega_0(t)}=\ip{V(t)}{V(t)}$, and the map $F_0:J\to\SLR$ by 
$$
\frac{\dd f_0}{\dd t}=e^{\omega_0/2}\Ad_{F_0}(e_0),\quad V=e^{\omega_0/2}\Ad_{F_0}(e_1),\quad F_0(t_0)=I.
$$

Assume that $f$ is a solution with a coordinate frame $F$, satisfying $F(u,0)=F_0(u)$, and an extended coordinate frame $\hat F$ constructed as in Section \ref{extendedframesection}. Along $\Delta$,  we then have 
\beqas
\hat F^{-1} \hat F_u &= &\hat U - \hat V\\
 &=&  \frac{1}{4} \bbar - \omega_x - \omega_y & -4 Q e^{-\omega/2} \lambda + 2 H e^{\omega/2} \lambda^{-1} \\
  2 H e^{\omega/2} \lambda - 4 R e^{-\omega/2} \lambda^{-1} & \omega_x + \omega_y \ebar.
\eeqas

At $\lambda =1$, this should agree with
$$
F_0^{-1} (F_0)_u = \bbar a & b \\ c & -a \ebar,
$$
where $a$, $b$ and $c$ are known functions on $\Delta$. Hence we have $Q-R = -\frac{1}{4} e^{\omega/2}(b-c)$ and $Q+R = \ip{f_{xx}+f_{yy}}{N}= \ip{f_{uu} + 2 f_{xy}}{N}=\ip{f_{uu}}{N} + H e^\omega$, which give us the formulae
\begin{equation}\label{eq:QR}
\begin{split}
Q &= -\frac{1}{8} e^{\omega/2}(b-c) + \frac{1}{2} \ipb{\frac{\dd^2 f}{\dd u^2}}{N}
  + \frac{1}{2}H e^\omega,\\
R &= \frac{1}{8} e^{\omega/2}(b-c) + \frac{1}{2} \ipb{\frac{\dd^2 f}{\dd u^2}}{N} + \frac{1}{2} H e^\omega.
 \end{split}
 \end{equation}
These are all known functions along $\Delta$, the normal $N$ being given by $\Ad_{F_0}(e_2)$, and the diagonal components of $\hat F^{-1}  \hat F_u$ along $\Delta$ are given by 
$a$ and $-a$ respectively. 

Thus, by defining $Q$ and $R$ along $\Delta$ by \eqref{eq:QR}, we obtain a map $\hat F_0:\Delta\to\G$ with $\hat F_0(x_0,x_0)=I$ and $\hat F_0(x,-x)\big\vert_{\lambda=1}=F_0(x)$. There is an open interval $J'$ containing $x_0$ with $\hat F_0(\Delta')\subset\B$, where $\Delta'=\{(x,y)\in\Delta\ |\ x\in J'\}$. Performing the left and right normalized Birkhoff decompositions on $\Delta'$ gives $\hat F_0 = \hat F^0_- \hat G^0_+ = \hat F^0_+ G^0_-$.
It follows from the construction of $\hat F_0$ that 
$$
(\hat F^0_-)^{-1} \dd \hat F^0_- = \psi_{-1} \lambda^{-1} \dd u,\qquad (\hat F^0_+)^{-1} \dd \hat F^0_+ = \chi_{1} \lambda^{1} \dd u,
$$
where $\psi_{-1}$ and $\chi_1$ do not depend on $\lambda$. Hence $\hat F_- (y) = \hat F^0_-(-y)$, and $\hat F_+ (x) = \hat F^0_+(x)$, correspond to a potential pair $\psi = \hat F_-^{-1} \dd \hat F_-$ and $\chi = \hat F_+^{-1} \dd \hat F_+$, on $I^\prime_x = \{x \in \real ~|~ (x,-x) \in  J^\prime \}$ and  $I^\prime_y = \{y \in \real ~|~ (-y,y) \in J^\prime \}$, respectively.  
 
As in  the second part of Theorem \ref{dpwthm}, define $\hat \Phi: I^\prime_x \times I^\prime_y \to \G$ by the expression $\hat \Phi(x,y) := \hat F_+(x)^{-1} \, \hat F_-(y)$; since $\hat \Phi(x_0, -x_0) = I$, there is an open set $W \subset M$, with $(x_0,-x_0) \in W$, such that $\hat \Phi (W) \subset \mathcal{B}_L$. Performing a left normalized Birkhoff splitting of $\hat \Phi$ on $W$ gives $\hat\Phi=\hat H_-\hat H_+$, and we have an admissible frame $\hat F=\hat F_-\hat H_+^{-1}: W \to \G$. 

By construction, this frame agrees with $\hat F_0(u)$ along $W \cap \Delta$, and is therefore regular along this set; hence, taking $W$ sufficiently small, regular on $W$. From the Sym formula we thus conclude that there is a solution defined in a neighbourhood of $(x_0,-x_0)$. 

\textbf{Global existence and uniqueness for item 1:} We have shown that we can cover $\Delta$ with open sets $W$ on each of which there is a solution of the geometric Cauchy problem.

Suppose $f: W \to \LL^3$ and $\tilde f: \tilde W \to \LL^3$ are two local solutions, with $W \cap \tilde W\cap\Delta\neq\emptyset$. Let $(x_0,-x_0)=(u_0,0)$ be a point in this set and $U\subset W\cap\tilde W$ a contractible neighbourhood of this point. Without loss of generality, we can assume that $\frac{\dd f_0}{\dd x}(x_0)=\alpha^2e_0$ and $V(x_0)=\alpha^2e_1$.  For each surface we have a unique extended frame,
respectively denoted by $\hat F$ and $\tilde{\hat F}$, with $\hat F(x_0,-x_0)=\tilde{\hat F}(x_0,-x_0)=I$. Since these are in the big cell $\B$ in a neighbourhood $U_1\subset U$ of $(x_0,-x_0)$, we may perform (normalized) left and right Birkhoff decompositions
$$
\hat F=\hat F_-\hat G_+=\hat F_+\hat G_-,\quad \hat{\tilde F}=\tilde{\hat F}_-\tilde{\hat G}_+=\tilde{\hat F}_+\tilde{\hat G}_-.
$$
It follows from the first part of the argument in the existence proof, that $\hat F_- = \tilde {\hat F}_-$ and $\hat F_+ = \tilde{\hat F}_+$ in $U_1\cap\Delta$, since these are completely determined from the geometric Cauchy data along $\Delta$.
Consequently $\hat F_-(-y) = \tilde {\hat F}_-(-y)$, for $(y,-y) \in U_1$, and  $\hat F_+(x) = \tilde{\hat F}_+(x)$ for $(x,-x) \in U_1$. Thus the normalized potentials for the two surfaces are identical here, and the surfaces therefore are identical on $U_1$. Hence we have a well-defined solution on some open set containing $\Delta$. It also follows from the argument just given that the solution is unique.\\
\noindent \textbf{Item 2:}
Now suppose that $\tilde f: \tilde M \to \LL^3$ is an arbitrary solution of the geometric Cauchy problem.
To show that it agrees with the solution in item 1, it is enough to show that, locally,
 there exists a choice of null coordinates $(x,y)$, with
corresponding isothermal coordinates $(u,v)$, such that the 
curve $f_0$ is given by $v=0$.  It then automatically follows that
$f_v(u,0)=V(u)$, because in such coordinates we can write $V=af_u+bf_v$; but then $a$ is zero because $f_u$ is tangent to the 
curve, to which $V$ is assumed to be orthogonal.  It then follows from
the condition $\left< V, V \right> = -\left< \dd f_0/\dd t, \dd f_0/\dd t \right>$ that $b=\pm 1$. Our assumption on the orientation
of the solution implies that $b=1$.

To show that the required coordinates exist, observe that, since 
the curve $f_0$ is timelike, it can,  in local null coordinates
$(x,y)$ for $\tilde M$, be expressed as a graph $y=h(x)$, where 
$h^\prime(x) < 0$. But then the coordinates $(\tilde x, \tilde y) = (-h(x), y)$ are also null, with the same orientation,  and the curve is given by $\tilde v= (\tilde x + \tilde y)/2=0$.
\end{proof}


\subsubsection{The boundary potential pair}  
The construction given in  the proof of Theorem \ref{gcptheorem} is not very practical,
as one needs to perform a Birkhoff decomposition to get the potential pair $(\chi, \psi)$. Below we show how to obtain an alternative potential pair directly from the geometric Cauchy data.  Again we describe only the case when $f_0$ is timelike; the analogue holds for a spacelike initial curve, using $\hat F_+(x) = \hat F_0(x)$ and  $\hat F_-(y) = \hat F_0(y)$, in lieu of the definitions below.

\begin{theorem} \label{bptheorem}
Assume that the geometric Cauchy data are given, and let $H\neq0$, $J^-$, $M= J \times J^-$
and $\Delta$ be as in in Theorem \ref{gcptheorem}.  Let $\hat F_0: J \to \gggg$
be the extended frame along $\Delta$ constructed as before. Set $\hat F_+(x) = \hat F_0(x)$ and $\hat F_-(y) = \hat F_0(-y)$. Then 
\begin{enumerate}
\item
$(\chi, \psi)= ( \hat F_+^{-1} \dd \hat F_+, \hat F_-^{-1} \dd \hat F_-)$
is a potential pair on $J \times J^-$, and is regular on an open set containing $\Delta$. 
\item
Set $\hat \Phi = \hat F_+^{-1} \hat F_- : J \times J^- \to \gggg$, and 
$M^\circ = \hat \Phi^{-1} (\mathcal{B}_L)$. The surface $f: M^\circ \to \LL^3$
obtained from the corresponding admissible frame $\hat F$ from Theorem \ref{dpwthm} is a solution of the given geometric Cauchy problem.
\end{enumerate}
\end{theorem}
\begin{proof}
Item (1) is clear from the construction.  To prove (2), we note that $\hat F$ is obtained by Birkhoff decomposing $\hat F_+^{-1} \hat F_- = \hat H_- \hat H_+$, and then setting $\hat F = \hat F_- \hat H_+^{-1}$.  Along $\Delta = \{v=x+y=0\}$, we have $u=x=-y$, so that $\hat F_+(x)^{-1} \hat F_-(y) = \hat F_0(x)^{-1} \hat F_0(-y) = I$, and thus $H_+ = I$ along $\Delta$.  Hence we have $\hat F(x,-x) = \hat F_-(-x) = \hat F_0(u)$ along  $\Delta$, which shows that $f$ is a solution of the geometric Cauchy problem.
\end{proof}

\subsection{The geometric Cauchy problem with null initial curve}
Next we consider the case where the initial curve $f_0$ is a null curve.

\noindent\textbf{Null geometric Cauchy data:} An open interval $J \subset \real$ containing $0$,  a regular smooth curve $f_0: J \to \LL^3$, which is everywhere null, that is $\left< \frac{\dd f_0}{\dd x} , \frac{\dd f_0}{\dd x} \right> = 0$, and a regular smooth \emph{null} vector field $V: J \to \LL^3$, such that $\left<  \frac{\dd f_0}{\dd x}, V \right> >0$ on $J$.

Note that the vector field $V$ carries more information than just specifying a family of timelike tangent planes along the curve, as $\frac{\dd f_0}{\dd x}$ and $V$ determine the conformal factor of the metric along the curve. However, even with this information, we do not have a unique solution in the null case. 

\begin{theorem}\label{th:GCPnull}
Let $I_y\subset\real$ be an interval containing $0$ and $\alpha,\beta:I_y\to\real$ two smooth functions, with $\alpha(0)\neq 0$. Let $H \neq 0$. Given null geometric Cauchy data on an interval $I_x$, together with the functions $\alpha$ and $\beta$, then  there is an open set $U\subset I_x\times I_y$, containing
the set $I_x \times \{0\}$, and a unique timelike CMC $H$ surface $f:U\to\LL^3$ with null coordinates $(x,y)$ such that 
\begin{equation}  \label{confconditions}
f(x,0) = f_0(x)\qquad\text{and}\qquad
    \frac{\partial f}{\partial y} (x,0) = V(x).
\end{equation}  
Conversely, given any local solution of the geometric Cauchy problem satisfying \eqref{confconditions}, there is a unique pair of such functions $\alpha$ and $\beta$ from which the solution is constructed. 
\end{theorem}

\begin{proof} Without loss of generality we may assume that $\frac{\dd f_0}{\dd x}(0)=\frac{1}{2}(e_0+e_1)$ and $V(0)=\frac{1}{2}(-e_0+e_1)$. 
Define a function $\omega_0:I_x\to\real$ by 
$$
\frac{e^{\omega_0}}{2}=\ipb{\frac{\dd f_0}{\dd x}}{V},
$$
and $F_0:I_x\to\SLR$ by 
$$
 \frac{\dd f_0}{\dd x} = \dfrac{1}{2} e^{\omega_0/2} \Ad_{F_0} (e_0 + e_1),\quad
V = \dfrac{1}{2} e^{\omega_0/2} \Ad_{F_0}(-e_0 + e_1),\quad F_0(0)=I.
$$
Let us write $F_0^{-1}(F_0)_x=ae_0+be_1+ce_2$ and set $\chi=(\lambda(ae_0+be_1)+ce_2)\dd x$. We integrate $\hat F_0^{-1} \dd \hat F_0 = \chi$ with initial condition $\hat F_0(0) = I$ to obtain a map $\hat F_0: J \to \G$. Clearly, if $f$ is any solution to the geometric Cauchy problem satisfying \eqref{confconditions}, and $\hat F$ is the extended coordinate frame with $\hat F(0,0) = I$, and coordinates chosen so that $\varepsilon_1=\varepsilon_2 = 1$, 
 then $\hat F(x,0) = \hat F_0(x)$. Thus, setting 
$$
\psi = \bbar 0 & \alpha \lambda^{-1} \\ \beta \lambda^{-1} & 0 \ebar \dd y,
$$
 then $(\chi, \psi)$ is a regular potential pair, and by Theorem \ref{dpwthm}, this pair corresponds to an admissible frame $\hat F$.  By construction, $\hat F(x,0) = \hat F_0(x)$, and the corresponding surface solves the geometric Cauchy problem. 

Finally, since $\psi$ is a \emph{normalized} potential, it is uniquely determined by an extended frame for the solutions of the geometric Cauchy problem. Thus, for any choice of $\psi$, there is a unique solution of the geometric Cauchy problem. 
\end{proof}

\begin{remark}
In the above proof, we have $\hat F(x,0) = \hat F_0(x)$ and $\hat F(0,y) = \hat F_-(y)$, where $\hat F_-^{-1} \dd \hat F_- = \psi$, and $\hat F_-(0) = I$. Thus it is necessary and sufficient to specify the geometric Cauchy data along \emph{two} (intersecting) null curves in order to obtain a unique solution to the problem in the null case.
\end{remark}

\begin{remark} We can write down the general solution quite explicitly: without loss of generality, we can assume that
\beqas
\frac{\dd f_0}{\dd x}(x)&=&\frac{1}{2}s(x)(e_0+ \cos(\theta(x))e_1 + \sin(\theta(x))e_2),\\
V(x)&=&\frac{1}{2}t(x)(-e_0+ \cos(\theta(x))e_1 + \sin(\theta(x))e_2),
\eeqas
where $s$, $t$ and $\theta$ are positive real-valued functions with $t(0)=s(0)=1$. Then $e^\omega=st$ 
and one finds
$$
F_0=\bbar r \cos(\theta/2) & -r^{-1}\sin(\theta/2) \\ r \sin(\theta/2) & r^{-1} \cos(\theta/2) \ebar,
$$ 
where $r=\sqrt{t/s}$. A simple calculation now gives
$$
\chi = \bbar \frac{r_x}{r} & -\frac{\theta_x}{2 r^2} \lambda \\ \frac{r^2 \theta_x}{2} \lambda & -\frac{r_x}{r}\ebar \dd x.
$$
A similar geometric expression can be given for $\psi$ in terms of data along $I_y$.
\end{remark}

\section{Surfaces of revolution}  \label{revsection}
Timelike surfaces of revolution in $\LL^3$ come in three types, according to the causal character of the axis which is fixed by the revolution.  The differential equations defining the possible profile curves for the three cases are given by R L\'opez \cite{lopez2000}. These are nonlinear and solutions have only been found by numerical methods.  

Below we will consider the problem of finding potential pairs for surfaces of revolution.  Knowledge of these potentials can be used to construct new examples of timelike CMC surfaces, as has been done for the case of CMC surfaces in $\real^3$
(see, for example, \cite{dorfwu2008, krs}).

We will work out the potentials for the null axis and timelike axis cases. The case with a 
spacelike axis can be done in a similar way.

\subsection{Rotations in $\LL^3$}
First we describe the $\SLR$ matrices corresponding to rotations about the three types of axes.  Without loss of generality, we can take the timelike axis  given by $e_0$, the spacelike by $e_1$, and the null axis by $e_0 +e_1$. It is easy to see that the rotation matrices, up to sign, are given respectively by 
$$
T(t)=\bbar \cos \frac{1}{2}t &  \sin \frac{1}{2}t \\ - \sin \frac{1}{2}t & \cos \frac{1}{2}t \ebar, \quad
S(t)=\bbar \cosh  \frac{1}{2}t &  \sinh  \frac{1}{2}t \\ \sinh  \frac{1}{2}t & \cosh  \frac{1}{2}t \ebar,\quad
L(t)=\bbar 1 & 0 \\ t & 1 \ebar.
$$

\subsection{Constructing the surfaces via the geometric Cauchy problem}

To find the potential pairs corresponding to all surfaces of revolution with a given type of axis, it is enough to solve the geometric Cauchy problem along a ``circle" generated by the corresponding rotation, with all possible choices of prescribed tangent plane. Since the surface is to be invariant under the relevant rotation, the vector field $V$ for the geometric Cauchy problem is determined at a single point on the circle.

\subsubsection{Timelike axis}

Let us consider the timelike axis in the direction of $e_0$.  Given a circle of radius $\rho$ centered on this axis, we may assume that this is given by 
$$
f_0(t) = \rho(\sin t \, e_1 + \cos t \,  e_2  ) =  \rho\Ad_{T (t)}(e_2).
$$
Since $f_0'(0)=\rho e_1$, we may assume that the vector field $V(t)$ for the geometric Cauchy problem satisfies $V(0)=\rho e_0$, so that $V(t)=\rho\Ad_{T(t)}(e_0)=\rho e_0$. Hence 
$$
e^{\omega_0}=\ipb{\frac{\dd f_0}{\dd t}}{\frac{\dd f_0}{\dd t}}=\rho^2.
$$
From the equations 
$$
\frac{\dd f_0}{\dd t}=e^{\omega_0/2}\Ad_{F_0}(e_1),\quad V(t)=e^{\omega_0/2}\Ad_{F_0}(e_0),\quad F_0(0)=I,
$$
we see that $F_0(t)=T(t)$. For a coordinate frame $F(u,v)$, in isothermal coordinates,
 for the solution of the geometric Cauchy problem, we should have $F(0,v)=F_0(v)$, so that  
$$
\frac{1}{2}\bbar 0 & 1 \\ -1 & 0 \ebar=U + V = \frac{1}{4} \bbar -\omega_x + \omega_y & -4 Q e^{-\omega/2} - 2 H e^{\omega/2} \\
  2 H e^{\omega/2} + 4 R e^{-\omega/2} & \omega_x - \omega_y \ebar,
$$
where $\omega(0,v)=\omega_0(v)$. Solving this, with $e^{\omega/2} = \rho$, we obtain $Q = R = -\frac{1}{2} \rho(1 + \rho H)$. 
Therefore, we set 
\begin{equation*}
\begin{split}
\hat A =& \bbar 0 & -Q \rho^{-1} \lambda  - \frac{1}{2}H \rho \lambda^{-1} \\
  \frac{1}{2}H \rho \lambda + R \rho^{-1} \lambda^{-1} & 0 \ebar \\
 =& \frac{1}{2} \bbar 0 & (1+ \rho H) \lambda - H \rho \lambda^{-1} \\
     H \rho \lambda - (1+\rho H) \lambda^{-1} & 0 \ebar,
\end{split}
\end{equation*}
and define $\hat F_0$ by integrating $\hat F_0^{-1}\dd\hat F=\hat A\dd v$ with $F_0(0)=I$. 
As discussed in Section \ref{dpwsection}, the boundary potentials are given 
by  
$$
\chi = \hat A \dd x,\qquad \psi = \hat A \dd y.
$$
Since $\hat A$ is constant in $x$ and $y$, these integrate to $\hat F_+(x)=\exp(\hat Ax)$ and $\hat F_-(y)=\exp(\hat Ay)$, respectively. Setting $\hat \Phi=\hat F_+(x)^{-1}\hat F_-(y)$ and performing a normalized Birkhoff decomposition of $\hat \Phi$ into $\hat H_-\hat H_+$, the extended frame for the solution is given by $\hat F=\hat F_+ \hat H_-$. 

\begin{figure}[here] 
\begin{center}
\includegraphics[height=26mm]{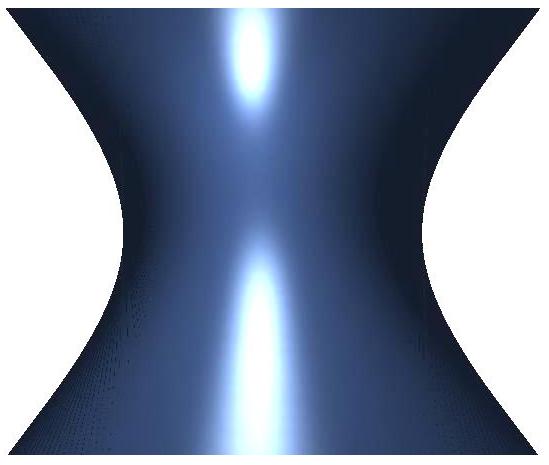}\hspace{1cm}
\includegraphics[height=26mm]{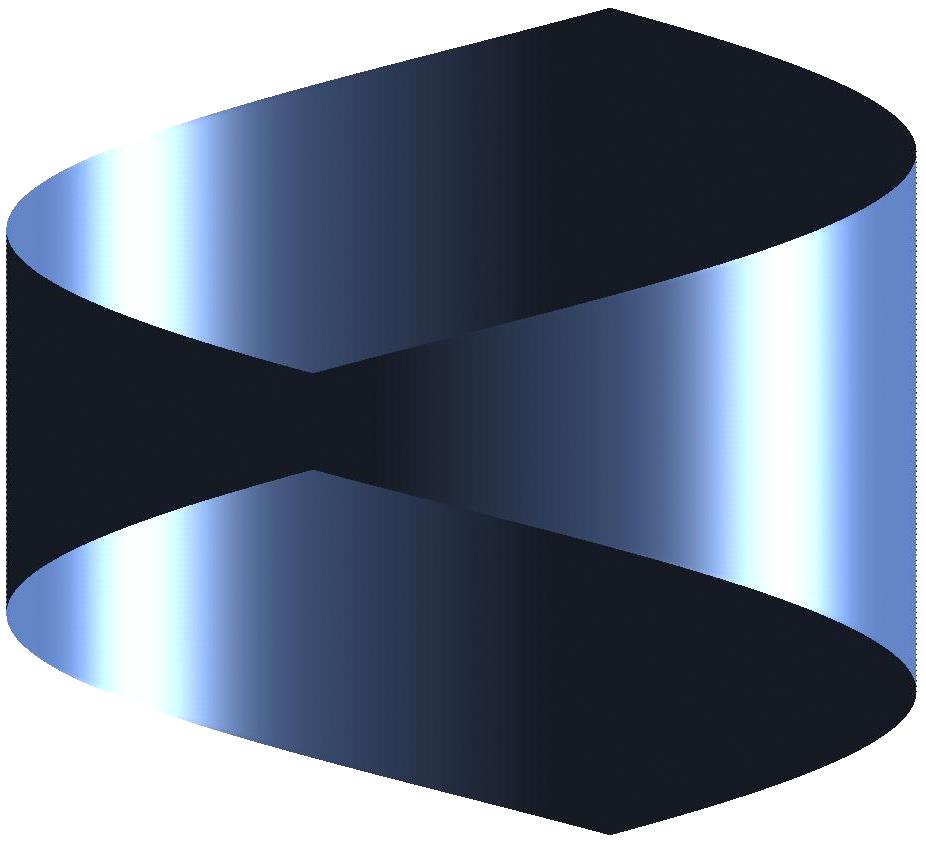} \hspace{1cm}
\includegraphics[height=26mm]{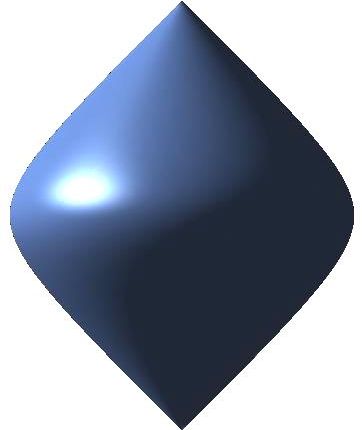} 
\end{center}
\caption{Partial plots of the typical timelike CMC surfaces of revolution with timelike axis, computed from the geometric Cauchy data on a circle of radius $\rho$. Left: $\rho H=-1$.  Center $\rho H=-1/2$.
Right: $\rho H=1$.  } \label{sorfig}
\end{figure}

\begin{example}
In general, the Birkhoff decomposition cannot be written down explicitly. 
The typical solutions have been computed numerically here (Figure \ref{sorfig}).
However, the special case when $\rho H=-1/2$, which is a transition point of
the topological type of the solution, can be worked out explicitly.
We have $\hat A = -\frac{1}{4} (\lambda +  \lambda^{-1})e_0$
and in this case,  
\beqas
\hat \Phi(x,y) &=& \exp\left(\frac{1}{4}(\lambda + \lambda^{-1})(x-y)e_0\right)\\
&=&\exp\left(\frac{\lambda^{-1}}{4}(x-y)e_0\right)\exp\left(\frac{\lambda}{4}(x-y)e_0\right)=\hat H_-\hat H _+.
\eeqas
Hence, we have the extended frame
$$
\hat F (x,y)=\hat F_-(y) \hat H_+^{-1} (x,y) = \bbar \cos \left(\frac{\lambda x + \lambda^{-1}y}{4}\right) &  \sin \left(\frac{\lambda x + \lambda^{-1}y}{4}\right) \\ - \sin \left(\frac{\lambda x + \lambda^{-1}y}{4}\right) & \cos \left(\frac{\lambda x + \lambda^{-1}y}{4}\right) \ebar.
$$
Then we can compute $\sym(\hat F) = \frac{1}{2}((y-x) e_0 - \sin(x+y) e_1 - \cos(x+y) e_2)$,
and since we are in the case $H= -\frac{1}{2\rho}$,  the surface is given by
\begin{eqnarray*}
f(x,y) & =& \frac{1}{2H} \sym_1 F(x,y) 
= \rho \frac{1}{2} ((x-y) e_0 + \sin(x+y) e_1 - \cos(x_y) e_2)\\
&=& \rho [ u, \sin (v), \cos(v)],
\end{eqnarray*}
a right circular cylinder in $\real^3$ of radius $\rho = -\dfrac{1}{2H}$.
\end{example}

\subsubsection{Null axis}

Consider now the null axis $e_0+e_1$, and a curve $f_0(t)$ which we can assume
is a null curve of the form
$$
f_0(t)=\Ad_{L(t)}(ae_0+be_1-ce_2), \quad a, ~b,~ c ~ \textup{constant}, ~c>0.
$$
Note that we cannot apply an isometry of $\LL^3$ to simplify $f_0$ and $V$ as
we did in the proof of Theorem \ref{th:GCPnull}, because this would move the
aixs of the rotation.

Since $f_0'(t)=\Ad_{L(t)}(c(e_0+e_1)+(-a+b)e_2)$, we see that $f_0$ is null if and only if $a=b$, so that $f_0'(t)=\Ad_{L(t)}(c(e_0+e_1))=c(e_0+e_1)$. To find a timelike CMC surface of revolution around the null axis containing this curve, we specify a null vector field $V(t)$ along this curve, invariant under $L(t)$. We can assume that it is given by 
$$
V(t)=\Ad_{L(t)}(A e_0+B e_1+C e_2), \quad A, ~ B, ~ C ~
  \textup{constant}, ~B-A\geq0.
$$
where we require $-A^2+B^2+C^2=0$. As before, set 
$$
\frac{e^{\omega_0}}{2}=\ipb{\frac{\dd f_0}{\dd t}}{V}=c(B -A). 
$$
Hence $\dfrac{e^{\omega_0/2}}{2}=\sqrt{\dfrac{c(B-A)}{2}}$. Next we find a map $F_0$ along the curve satisfying 
$$
f_0'(t)=\frac{e^{\omega_0/2}}{2}\Ad_{F_0(t)}(e_0+e_1),\quad V(t)=\frac{e^{\omega_0/2}}{2}\Ad_{F_0(t)}(-e_0+e_1).
$$
Writing $F_0(t)=\begin{pmatrix} x & y \\ z & w \end{pmatrix}$, we see that
$$
c(e_0+e_1)=\frac{e^{\omega/2}}{2}\Ad_{F_0(t)}(e_0+e_1)=2\sqrt{\frac{c(B-A)}{2}}\begin{pmatrix} yw & -y^2 \\ w^2 & -yw \end{pmatrix}.
$$
Thus, $y=0$ and $x^{-1}=w=\pm (2c/(B-A))^{1/4}$.
 We may choose the plus sign here, and, after scaling the coordinate of the curve, we may also assume that $2c=B-A$, so that $x=w=1$.
Finally, from the two expressions for $V(t)$, we obtain $z=t+C/(2c)$. A translation 
of the coordinate $t$, allows us to take $C=0$, so that
$$
F_0(t)=\begin{pmatrix} 1 & 0 \\ 
t  & 1  \end{pmatrix}.
$$
which gives, 
$$
F_0^{-1}(F_0)_t=\begin{pmatrix} 0 & 0 \\ 1 & 0 \end{pmatrix}.
$$
As in Theorem \ref{th:GCPnull}, we multiply by $\lambda$ and exponentiate to obtain 
$\hat F_+(x)$, and  all solutions are obtained by choosing two arbitrary 
functions $\alpha(y)$ and $\beta(y)$, with $\alpha(0)\neq 0$, and 
the potential pair:
$$
\chi = \begin{pmatrix} 0 & 0 \\  \lambda & 0 \end{pmatrix} \dd x,
 \quad \psi = \lambda^{-1}\begin{pmatrix} 0 & \alpha \\ \beta & 0 \end{pmatrix}\dd y.
$$

\begin{example}
Taking $\alpha=1$ and $\beta=0$, gives $\hat F_-(y)=\begin{pmatrix} 1 & \lambda^{-1}y \\ 0 & 1 \end{pmatrix}$. In the (left) normalized Birkhoff decomposition $\hat F_+^{-1}\hat F_-=\hat H_-\hat H_+$, it follows easily that 
$$
\hat H_-=\begin{pmatrix} 1 & \frac{y\lambda^{-1}}{1- xy} \\ 0 & 1 \end{pmatrix}.
$$
Hence 
$$
\hat F=\hat F_+\hat H_-=\begin{pmatrix} 1 & \dfrac{y\lambda^{-1}}{1- xy} \\ \lambda  x & \dfrac{1}{1- xy} \end{pmatrix}.
$$
From the Sym formula, (choosing $H=1/2$) we get the surface 
$$
f(x,y)=\frac{1}{1- xy}(2( x+y)e_0+2( x-y)e_1-(3 xy+1)e_2).
$$
\end{example}

\section{Pseudospherical surfaces in Euclidean 3-space}  \label{ksurfsection}

\subsection{The loop group formulation}

Let $D\subset\real^2$ be a simply connected domain. An immersion $f:D\to\E^3$ is said to be a  \emph{pseudospherical surface} if it has constant sectional curvature $-1$. Following Bobenko \cite{bobenko1994}, let $x$ and $y$ be asymptotic coordinates for $f$, not necessarily arc length coordinates. Let $\phi$ denote the oriented angle between
$f_x$ and $f_y$.
The first and second fundamental forms of $f$ are given by  
$$
I=|f_x|\dd x^2+2\cos(\phi)\dd x\dd y+|f_y|\dd y^2,\quad II=2|f_x||f_y|\sin(\phi)\dd x\dd y,
$$
and the equations of Gauss and Codazzi-Mainardi reduce to
$$
\phi_{xy}-|f_x||f_y|\sin(\phi)=0,\quad \partial_y|f_x|= \partial_x|f_y|=0.
$$
Set $\theta=\phi/2$ and 
\begin{equation}\label{eq:E1E2}
E_1 = \frac{1}{2 \cos(\theta)}\biggl(\frac{f_x}{|f_x|} + \frac{f_y}{|f_y|}\biggr), \quad E_2 = -\frac{1}{2 \sin(\theta)}\biggl(\frac{f_x}{|f_x|}- \frac{f_y}{|f_y|}\biggr).
\end{equation}
 It is easy to see that $E_1$ and $E_2$ are unit principal vector fields for the surface. 
 
To obtain a loop group formulation, we identify $\E^3$ with the Lie algebra $\su$ with the orthonormal basis 
\bdm
e_1=\frac{1}{2}\bbar 0 & i \\ i & 0\ebar, \quad
e_2 =\frac{1}{2}\bbar 0 & -1 \\ 1 & 0\ebar, \quad
e_3 = \frac{1}{2}\bbar i & 0 \\ 0 & -i\ebar.
\edm
A \emph{Darboux frame} for $f$ is a map $F:D\to\SU$ adapted to the principal directions in the sense that
\beq \label{ksurfcoordframe}
E_1 = \Ad_F e_1, \quad E_2 = \Ad_F e_2,  \quad N = \Ad_F e_3.
\eeq
We assume that coordinates are chosen so that $F(0,0)=I$. Let $U =F^{-1} F_x $ and $V = F^{-1} F_y$. A simple calculation shows that 
\beq \label{ksurfconnection}
 U  = \frac{i}{2} \bbar -\theta_x &  |f_x|e^{-i\theta}  \\ 
 |f_x| e^{i\theta}   & \theta_x \ebar, \quad
V = \frac{i}{2} \bbar \theta_y & -|f_y| e^{i \theta}   \\ -|f_y| e^{-i \theta}  & -\theta_y \ebar.
\eeq
Let us introduce the matrices
\beq  \label{mcfksurf}
\hat U  = \frac{i}{2} \bbar -\theta_x &  |f_x| e^{-i\theta} \lambda \\ |f_x| e^{i\theta} \lambda   & \theta_x \ebar, \quad \hat V = \frac{i}{2} \bbar \theta_y & -|f_y| e^{i \theta} \lambda^{-1}  \\ -|f_y| e^{-i \theta} \lambda^{-1} & -\theta_y \ebar,
\eeq
where $\lambda\in\C\setminus\{0\}$. The equations of Gauss and Codazzi-Mainardi imply that $-\hat U_x+\hat V_y+[\hat U, \hat V]=0$, and integrating $\hat F^{-1} \dd \hat F = \hat U \dd x + \hat V \dd y$ with $\hat F(0,0)=I$, we obtain an \emph{extended Darboux frame}, that is, a map $\hat F:D\to\Lambda\SLC_{\sigma \rho}$ with $\hat F\big\vert_{\lambda=1}=F$. Here $\sigma$ is the twisting involution defined previously, and $\rho$ is defined by
\bdm
(\rho \gamma)(\lambda) = \left(\overline{\gamma(\bar \lambda)}^T\right)^{-1}.
\edm
From $\hat F$ we obtain from the \emph{Sym formula} a family of pseudospherical surfaces
\beq \label{sym2}
f^\lambda = \lambda\frac{\partial \hat F}{\partial \lambda} \hat F^{-1}\qquad(\lambda\in\real\setminus\{0\}).
\eeq
It follows easily that $f^1$ coincides with $f$ up to a rigid motion of $\E^3$. 

\begin{definition}  \label{admissibleframedef2}
Let $M$ be a simply connected subset of $\real^2$, and let $(x,y)$ denote the standard 
coordinates. An \emph{admissible frame} is a smooth map $\hat F: M \to \Lambda\SLC_{\sigma \rho}$  such that the Maurer-Cartan form of $\hat F$ has the form
\begin{equation}  \label{admissible-K}
\hat F^{-1} \dd \hat F =\lambda\, A_{1}\, \dd x + \alpha_0 + \lambda^{-1}A_{-1}\, \dd y,
\end{equation}
where the coefficients $A_1$ and $A_{-1}$ and the $\su$-valued 1-form $\alpha_0$ are
constant in $\lambda$.  The admissible frame $\hat F$ is said to be \emph{weakly regular} if $[A_1]_{12}\neq 0$ and $[A_{-1}]_{12}\neq 0$. The frame is said to be \emph{regular} if 
it is weakly regular and $\textup{Arg}([A_1]_{12}) - \textup{Arg}([A_{-1}]_{12}) \neq k \pi$ for any integer $k$.
\end{definition}

Note that the reality condition given by $\rho$ means that the matrices $A_i$ are in $\mathfrak{su}(2)$. 
Hence, for a weakly regular frame, all the off-diagonal 
components of $A_{\pm1}$ are non-zero.  Given a weakly regular admissible frame $\hat F$, 
it is straightforward to verify (as in the proof of Proposition \ref{characterizationprop})
that by multiplying $\hat F$ on
the right with a  matrix-valued function of the form $T = \textup{diag}(e^{i\mu}, e^{-i \mu})$, we can 
bring the Maurer-Cartan form $\hat F^{-1} \dd \hat F$ 
into the form of (\ref{mcfksurf}),
with $2 \theta \in [0,\pi)$. If the frame is \emph{regular}, then one also has that 
$2 \theta \in (0,\pi)$.   In this case the analogue of Proposition \ref{characterizationprop} holds, in the sense that regular admissible frames correspond precisely 
to pseudospherical surfaces.

\begin{definition} \label{potentialdefn2}
Let $I_x$ and $I_y$ be two real intervals, with coordinates $x$ and $y$, respectively. A \emph{potential pair} $(\chi, \psi)$ is a pair of smooth 
$\Lambda\slC_{\sigma \rho}$-valued 1-forms on $I_x$ and $I_y$ respectively with Fourier expansions in $\lambda$ as follows:
\begin{eqnarray*}
\chi = \sum_{j=-\infty}^1 \chi_i \lambda^i \dd x,\qquad \psi = \sum_{j=-1}^\infty \psi_i \lambda^i \dd y.
\end{eqnarray*}
The potential pair is called \emph{weakly regular} if $[\chi_1]_{12}\neq 0$ and $[\psi_{-1}]_{12}\neq0$.
\end{definition}
Note that, again, if the pair is weakly regular then one also has
 $[\chi_1]_{21}\neq 0 \neq [\psi_{-1}]_{21}$.  It is straightforward
 to verify that the analogue of Theorem \ref{dpwthm} holds in this situation, with the only essential difference being that, as is shown in  \cite{jgp},  the big cell is the whole group for $\Lambda\SLC_{\sigma \rho}$ so that $\Phi^{-1}(\mathcal{B}_L)$ is
 the whole of $M$. However, a weakly regular potential pair only
 produces a weakly regular admissible frame;  there is no guarantee that 
 the corresponding pseudospherical surface is everywhere regular.

\subsection{Solution of the geometric Cauchy problem}  \label{ksurfsolsection}
As with timelike CMC surfaces, we will find that the problem splits into two quite different
situations, that of non-characteristic curves (non-asymptotic curves) and characteristic 
curves (asymptotic curves).  The results are broadly similar to the timelike CMC case, but 
with important differences.

\subsubsection{The non-characteristic case}
Recall that a curve $f_0$ on a surface with unit normal $N$ is asymptotic if
and only  $\ip{f_0'}{N'}=0$. For the first case that we consider we assume that the following is given:\\
\noindent \textbf{Non-characteristic Geometric Cauchy data:} An open interval $J \subset \real$ containing $0$, 
a  regular smooth map $f_0: J \to \E^3$,
 and a regular smooth vector field $N_0: J \to  \E^3$, which is everywhere orthogonal to $f_0^\prime$, and such that 
$$
\ip{f_0'}{N_0'}\neq0.
$$
\\
\noindent \textbf{The Geometric Cauchy problem:} Find a pseudospherical surface $f: D \to \E^3$, where $D$ is an open subset of $\real^2$
containing $J$, such that  $f \big|_J = f_0$ and the normal to the surface along $J$
is given by $N_0$.\\

The non-characteristic problem splits further into two subcases; however these are 
\emph{not} analogous to the split into timelike/spacelike cases encountered 
earlier. The two cases are when $f_0'$ and $N_0'$ are everywhere parallel, and when $f_0'$ and $N_0'$ are nowhere parallel, in other words when $f_0$ is everywhere or is never tangent to a
principal direction.
 

Given a pseudospherical surface $f:D\subset\real^2\to\E^3$, with coordinates and Darboux frame as described above, let us define new coordinates by
\bdm
u=\frac{1}{2}(x-y), \quad v=\frac{1}{2}(x+y).
\edm
Then 
\begin{equation}\label{eq:Fu}
F^{-1}F_u=U-V=\frac{i}{2}\begin{pmatrix} -\theta_v & |f_x|e^{-i\theta}+|f_y|e^{i\theta} \\ |f_x|e^{i\theta}+|f_y|e^{-i\theta} & \theta_v\end{pmatrix}.
\end{equation}
where, by \eqref{ksurfconnection}, we have 
\begin{equation}\label{eq:thetav}
\theta_v=\ip{(E_2)_u}{E_1}.
\end{equation}

Let $\alpha=|f_x|-|f_y|$ and $\beta=|f_x|+|f_y|$. From \eqref{eq:E1E2} we have 
$$
f_u=f_x-f_y=\alpha\cos(\theta) E_1-\beta\sin(\theta) E_2
$$
and from \eqref{ksurfconnection} we have 
$$
N_u=\Ad_F([U-V,e_3])=-\alpha\sin(\theta) E_1-\beta\cos(\theta) E_2.
$$
Hence we have 
\beqas
\ip{f_u}{N_u}=\frac{(\beta^2-\alpha^2)}{2}\sin(\phi),\\ |N_u|^2-|f_u|^2=(\beta^2-\alpha^2)\cos(\phi),\\
 |N_u|^2+|f_u|^2=\beta^2+\alpha^2.
\eeqas
It follows that 
$$
|f_u|^2|N_u|^2-\ip{f_u}{N_u}^2=2\alpha^2\beta^2.
$$
As $\beta>0$, we see that $\alpha=0$ if and only if $f_u$ and $N_u$ are parallel, that is, if and only if the curve $v=0$ is a principal curve on the surface. If this is the case, then 
\beqa\label{eq:beta1}
\beta= 2|f_x| = 2 |f_y|=  \sqrt{|f_u|^2+ |N_u|^2},\\
\label{eq:phi1}
\sin(\phi) = \frac{2\ip{f_u}{N_u}}{|f_u|^2+ |N_u|^2},
\quad
\cos(\phi) = \frac{|N_u|^2-|f_u|^2 }{|f_u|^2+ |N_u|^2}.
\eeqa
From the expressions $E_2=-(\beta\sin(\theta))^{-1}f_u$ and $E_1=E_2\times N$ we obtain 
\begin{equation}\label{eq:thetav2}
\theta_v=\ip{(E_2)_u}{E_1} =\frac{1}{|f_u|^2}\ip{f_{uu}}{f_u\times N}.
\end{equation}
Thus, we see that the matrices $U$ and $V$ in \eqref{ksurfconnection} may be obtained along the curve $v=0$ entirely from knowledge of $f(u,0)$ and $N(u,0)$. 

On the other hand, if $f_u$ and $N_u$ are never parallel along $v=0$, equivalently, if $\alpha$ is non-zero everywhere along this curve, then we may assume that $\alpha>0$. Then we have
$$
E_1=\frac{1}{\alpha}(\cos(\theta) f_u-\sin(\theta) N_u),\quad E_2=-\frac{1}{\beta}(\sin(\theta) f_u+\cos(\theta) N_u), 
$$
and subsequently obtain 
\begin{gather*}
\tan(\phi) = \frac{2\ip{f_u}{N_u}}{|N_u|^2 - |f_u|^2},\\
\alpha^2 = \frac{1}{2}\left(|f_u|^2(1+\cos^{-1}(\phi)) + |N_u|^2(1-\cos^{-1}(\phi))\right), 
 \label{Aeqn}\\
\beta^2 =  \frac{1}{2}\left(|f_u|^2(1-\cos^{-1}(\phi)) + |N_u|^2(1+\cos^{-1}(\phi))\right).
\end{gather*}
Let us also assume that $\ip{f_u}{N_u}\neq0$; by changing the sign of $N$ if necessary, we may assume that $\ip{f_u}{N_u}>0$. With this assumption $\phi$ is a well defined  function with values in $(0,\pi)$, satisfying
\beq  \label{phieqn}
\cos(\phi) = \frac{Z}{\sqrt{1+Z^2}}, \quad \sin(\phi) = \frac{1}{\sqrt{1+Z^2}},\quad Z = \frac{|N_u|^2-|f_u|^2|}{2\ip{f_u}{N_u}}.
\eeq
Noting that $\cos(\phi)$ has the same sign as $(|N_u|^2-|f_u|^2$) and recalling that $\alpha>0$ and $\beta>0$, we obtain
\begin{eqnarray}
\begin{split}
\label{alphabetaeqns}
 \alpha = \frac{1}{\sqrt{2}}\left( |f_u|^2 + |N_u|^2 -\sqrt{4\ip{f_u}{N_u}^2 + \left(|N_u|^2-|f_|^2 \right)^2}\right)^{1/2},\\
\beta = \frac{1}{\sqrt{2}}\left( |f_u|^2 + |N_u|^2 +\sqrt{4 \ip{f_u}{N_u}^2 + \left(|N_u|^2-|f_u|^2\right)^2}\right)^{1/2}.
\end{split}
\end{eqnarray}
As before, we have $\theta_v = \ip{(E_2)_u}{E_1}$. From the expressions above for $E_1$ and $E_2$ and from $\theta_u=-Z_u/2(1+Z^2)$, we obtain 
\beq \label{thetaveqn}
 \theta_v = \frac{\alpha Z_u}{2 \beta (Z^2+1)} - \frac{Y}{\alpha \beta},
\eeq
where
$$
Y = \sin(\theta)\cos(\theta)\left(\ip{f_{uu}}{f_u} - \ip{N_u}{N_{uu}}\right) +  \cos^2(\theta)\ip{f_u}{N_{uu}} -  \sin^2(\theta)\ip{N_u}{f_{uu}}.
$$

It follows that, as in the case $\alpha=0$,  we can construct the matrices $U$ and $V$ in \eqref{ksurfconnection} from knowledge of $f(u,0)$ and $N(u,0)$. The following  result now follows easily.

\begin{theorem} \label{ksurftheorem1}
Let the non-characteristic geometric Cauchy data $f_0$ and $N_0$ be given, where 
$f_0^\prime$ and $N_0^\prime$ are either everywhere parallel, or nowhere parallel.
In the first case set:
\beqas
\alpha(t)= 0, \quad \beta(t) = \sqrt{|f_0^\prime|^2+|N_0^\prime|^2},
\quad
\theta_v(t)=\frac{1}{|f_0^\prime|^2}\ip{f_0^{\prime \prime}} {f_0^\prime \times N_0},\\
\cos(\theta(t))= \frac{|N_0^\prime|}{\sqrt{|f_0^\prime|^2+|N_0^\prime|^2}}, \quad
\sin(\theta(t))= \frac{|f_0^\prime|}{\sqrt{|f_0^\prime|^2+|N_0^\prime|^2}}.
\eeqas
In the second case, substitute $f_0(t)$, $N_0(t)$ and $t$ for $f$, $N$ and $u$
in the expressions at \eqref{phieqn},
\eqref{alphabetaeqns}  and \eqref{thetaveqn}, to find the expressions
for $\alpha(t)$, $\beta(t)$, $\theta(t) := \phi(t)/2$ and $\theta_v(t)$.
In either case define $\hat A(t)$ to be the loop algebra valued function:
$$
 \frac{i}{2}\bbar -\theta_v & \frac{1}{2}(\beta+\alpha)e^{-i\theta}\lambda + \frac{1}{2}(\beta-\alpha)e^{i\theta}  \lambda^{-1}\\
\frac{1}{2}(\beta+\alpha)e^{i\theta}\lambda + \frac{1}{2}(\beta-\alpha)e^{-i \theta}  \lambda^{-1} & \theta_v \ebar.
$$
Then
\begin{enumerate}
\item
The pair of 1-forms $(\chi, \psi) := (\hat A(x) \dd x, -\hat A(-y) \dd y)$ is a 
weakly regular potential pair on $J \times J^-$, where $J^-=\{y \in \real| -y \in J\}$.
\item
The pseudospherical surface obtained from $(\chi, \psi)$ is regular on an open set containing
$\Delta:= \{(x,-x)\}\subset J \times J^-$. It is the unique solution of the given geometric Cauchy problem.
\end{enumerate}
\end{theorem}
The uniqueness in item (2) should be understood in the sense described in Section 
\ref{resultssection}. 

The proof of this theorem is essentially identical to the proof of the corresponding
result for timelike CMC surfaces, see Theorems \ref{gcptheorem} and \ref{bptheorem}.
The only difference here is that the weakly regular solution is defined on the whole of 
$J\times J^-$, as the big cell is the whole group. By construction, $\phi \in (0,\pi)$
along $\Delta$, which implies regularity on an open set containing $\Delta$.

\begin{remark}\label{remarksolution}
The expression (\ref{thetaveqn}) for $\theta_v$ 
implies that the solution for the case $\alpha \neq 0$, with the choice of coordinates we seek, will not exist at 
 points where $\alpha$ does vanish, 
unless we also have the condition that $\frac{Y}{\alpha \beta}$ 
is bounded. Since $\alpha \beta = 2(|f_u|^2|N_u|^2 - \ip{f_u}{N_u}^2)$, whilst $Y$ contains higher order derivatives, 
one would have to impose a condition on the second derivatives of $f$ and $N$ in order
to extend the theorem to the situation where $\alpha$ is non-zero at some points but zero at others.  Conditions on the second derivatives of $f$ and $N$ 
also arise in the expression for $\alpha_u$, which needs to be smooth when $\alpha$ vanishes.
\end{remark}


\subsubsection{The characteristic case}  \label{asymptsection}
Now we consider the case of an asymptotic curve, that is $\ip{f_0'}{N_0'}=0$ along $J$.\\
\noindent \textbf{Characteristic Geometric Cauchy data:} An open interval $J \subset \real$ containing $0$, 
a  regular smooth map $f_0: J \to \E^3$,
 and a regular smooth vector field $N_0: J \to  \E^3$, which is everywhere orthogonal to $f_0^\prime$, and such that $f_0^\prime$ and $N_0^\prime$ are also orthogonal.
$$
\ip{f_0'}{N_0'}=0.
$$
Since the curve is necessarily an asymptotic curve of any solution to the geometric Cauchy
problem, we are looking for a solution such that $f(x,0)= f_0(x)$.

Consider a pseudospherical surface $f$ parameterized by asymptotic coordinates $(x,y)$.
By multiplying a Darboux frame for $f$ on the right by the matrix function $\textup{diag}(e^{-i\theta/2},e^{i\theta/2})$, we obtain a new frame $F$ satisfying
$$
\frac{f_x}{|f_x|} =\Ad_F(e_1),\quad \frac{1}{\sin(2\theta)}\left(-\cos(2\theta)\frac{f_x}{|f_x|}+\frac{f_y}{|f_y|}\right)=\Ad_F(e_2).
$$
A simple calculation shows that 
$$
F^{-1}\dd F=\frac{i}{2}\bbar -\phi_x & |f_x| \\ |f_x| & \phi_x \ebar\dd x+\frac{i}{2}\bbar 0 & -|f_y|e^{i\phi} \\ -|f_y|e^{-i\phi} & 0 \ebar\dd y.
$$
Hence, if 
$$
\hat U=\frac{i}{2}\bbar -\phi_x & |f_x|\lambda \\ |f_x|\lambda & \phi_x\ebar,\quad \hat V=\frac{i}{2}\bbar 0 & -|f_y|e^{i\phi}\lambda^{-1} \\ -|f_y|e^{-i\phi}\lambda^{-1} & 0 \ebar,
$$
we may integrate $\hat F^{-1}\dd \hat F=\hat U\dd x+\hat V\dd y$, with $\hat F(0,0)=I$. Clearly, $\hat F$ will be an extended Darboux frame for $f$ up to multiplication from the right by a diagonal matrix (which has no effect on the Sym formula). 

To return to the geometric Cauchy problem, assume that characteristic geometric Cauchy data $f_0$ and $N_0$ are given. Solving $f_0'/|f_0^\prime|=\Ad_{F_0}(e_1)$,
 $N_0=\Ad_{F_0}(e_3)$ with $F_0(0)=I$ (after a suitable isometry of $\E^3$),
 we see from the condition that $\ip{f_0'}{N_0'}=0$ that 
$$
F_0^{-1}\dd F_0=\bbar ia & ib \\ ib & -ia\ebar\dd x,
$$
where $a$ and $b$ are two real-valued functions on $J$.  Thus, we define the potential 
$$
\chi=\bbar ia & ib\lambda \\ ib\lambda & -ia\ebar\dd x.
$$
To find a solution, we may now freely choose a potential of the form 
$$
\psi=\bbar 0 & \alpha\lambda^{-1} \\ \bar\alpha\lambda^{-1} & 0 \ebar\dd y,
$$
where $\alpha:I_y\to\C$ is a non-vanishing function.
  It is now not difficult
to see that the analogue of Theorem \ref{th:GCPnull} holds in this case also, where 
in this case
the additional data required is a non-vanishing function $\alpha: I_y \to \C$. 
The potential pair $(\chi, \psi)$ given above
defines  a solution to the problem, where $\chi$ is constructed from the Cauchy data,
and $\psi$ is determined by $\alpha$, and conversely every solution corresponds to such a
function $\alpha$.


\subsection{Applications}
As a consequence of Theorem \ref{ksurftheorem1} we have:
\begin{corollary} \label{ksurfcor}
Let $\gamma: J \to \E^3$ be a regular, constant-speed curve with  curvature
 $\kappa$,
and torsion $\tau$.  Suppose that $\kappa$ is nonvanishing and 
that either $\gamma$ is a plane curve or 
$\tau$ is nonvanishing. Then there exists a unique regular pseudospherical surface which contains $\gamma$ as a 
geodesic and this surface can be constructed by Theorem \ref{ksurftheorem1}.
\end{corollary}
\begin{proof}
This follows by applying Theorem \ref{ksurftheorem1}, choosing $f_0$ to be $\gamma$ and $N_0$ to be
the principal unit normal to the curve.
\end{proof}
Note  that a geodesic curve $\gamma$ is a principal curve if and only if it is
a plane curve.
\begin{figure}[here] 
\begin{center}
\includegraphics[height=30mm]{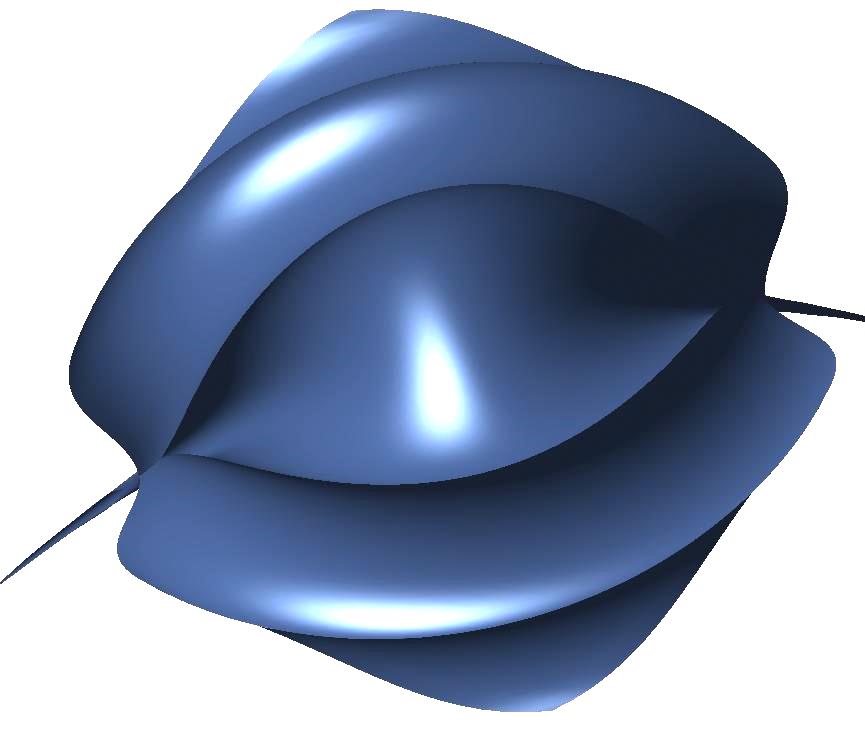} \hspace{0.4cm}
\includegraphics[height=30mm]{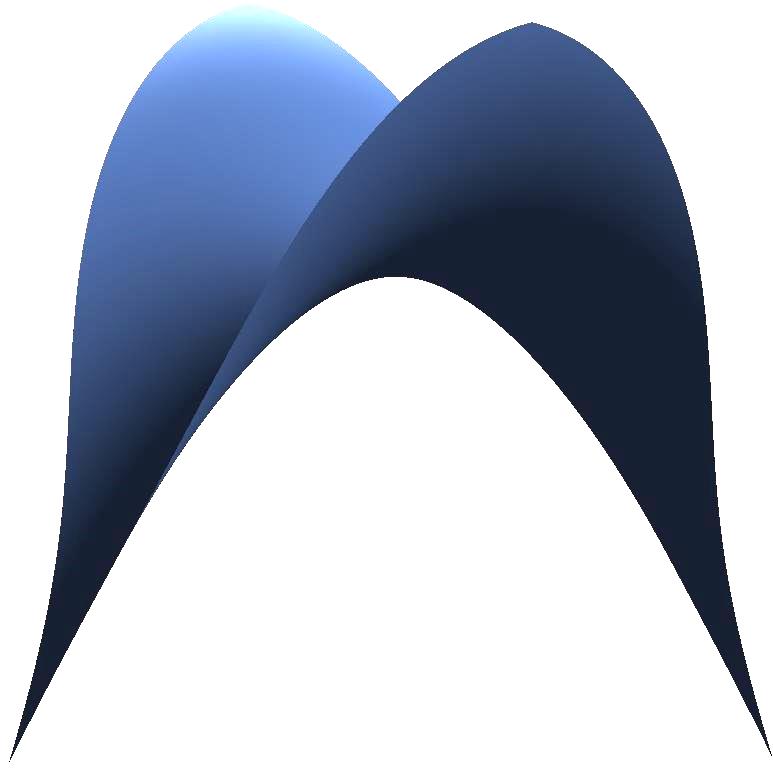} \hspace{0.4cm}
\includegraphics[height=30mm]{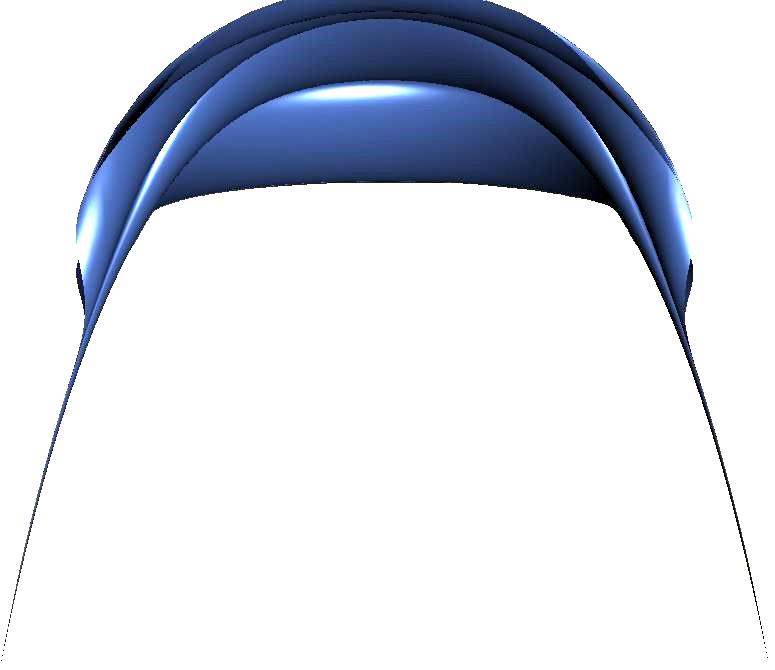}
\end{center}
\caption{Left, center: two plots of a part of the unique constant Gauss  curvature $-1$ surface which contains the parabola $y=x^2$ as a geodesic principal curve. 
Right the pseudospherical surface that contains the catenary $y=\cosh(x)$ as a geodesic principal curve. }  \label{parabolafigure}
\end{figure}
\begin{example}
We show how to construct the unique pseudospherical surface which contains the parabola $y=x^2$ as a geodesic. Take the curve $f_0(t)=(t,t^2,0)$ and $N_0(t)=\frac{1}{\sqrt{1+4t^2}}(2t,-1,0)$; since $N_0$ is the principal normal to $f_0$, $f_0$ will necessarily be a geodesic on the resulting pseudospherical surface. Note that $N_0^\prime(t)= (1+4t^2)^{-3/2}(2,4t,0)$ is parallel to $f_0'$, and $f_0$ will also be a principal curve on the surface. 
The solution will be well defined along the entire parabola
because $\left<f_0^\prime(t),N_0^\prime(t)\right>=(2+8t^2)(1+4t^2)^{-3/2}>0$ 
for all $t$. 
From the  formulae in Theorem \ref{ksurftheorem1} we have 
\beqas
|f_x| = |f_y|= \frac{\sqrt{(1+4t^2)^3+4}}{2(1+4t^2)}, \quad \theta_v =0\\ 
e^{i\theta} = \frac{2}{\sqrt{4+(1+4t^2)^3}}+ i \sqrt{\frac{(1+4t^2)^3}{4+(1+4t^2)^3}}.
\eeqas
The surface, together with the surface similarly generated by the catenary
$f_0(t)=(t,\cosh(t),0)$, can be seen in Figure \ref{parabolafigure}.

\begin{figure}[here] 
\begin{center}
\includegraphics[height=50mm]{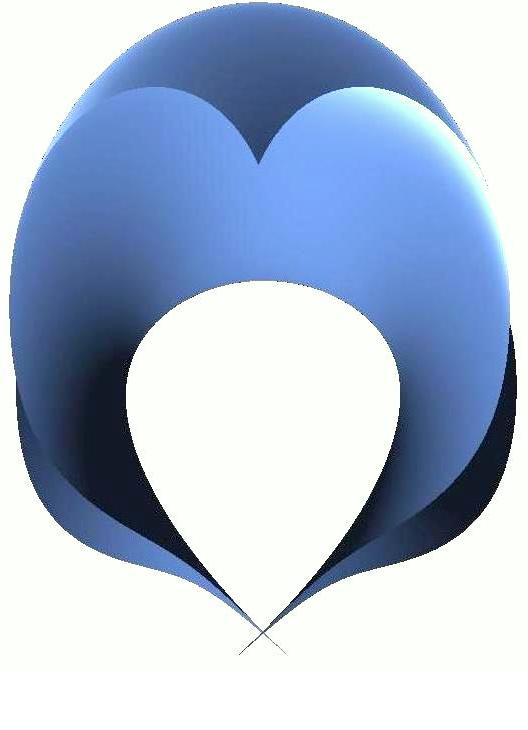} \hspace{1.5cm}
\includegraphics[height=50mm]{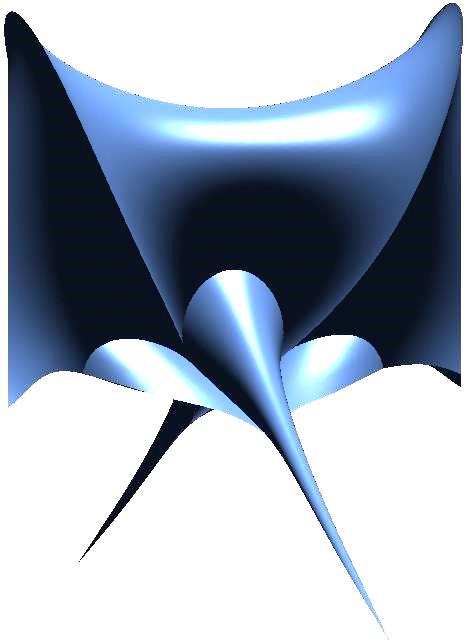}  \\
\end{center}
\caption{Two views of the unique pseudospherical surface that contains the cubic $y^2=x^2(x+1)$ as a geodesic principal curve. The cusp lines converge at infinity. }  \label{cubicfigure}
\end{figure}

\begin{figure}[here] 
\begin{center}
\includegraphics[height=50mm]{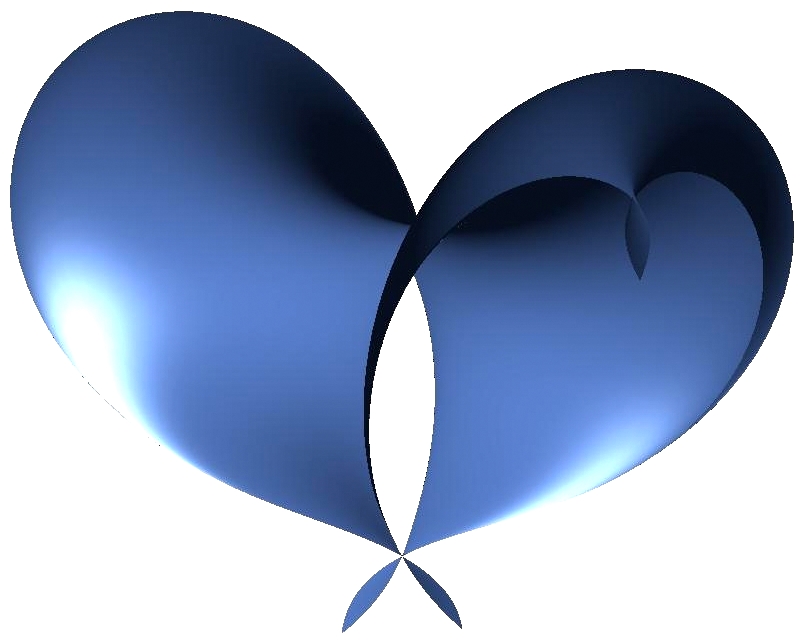}
\end{center}
\caption{The surface generated
by Bernoulli's Lemniscate $(x^2+y^2)^2 = x^2-y^2$. The cusp lines meet at the center of the figure eight, where the curvature of the plane curve is zero (compare Figure \ref{cubicfigure}).}  \label{lemfigure}
\end{figure}

In the same way, taking the parameterizations
\beqas
f_0(t)=(t^2-1, t(t^2-1),0), \\
f_0(t)= (\cos(t)/(1+\sin(t)^2), \sin(2t)/(2(1+\sin(t)^2)), 0),\\
f_0(t)=(\sin(t), 2\cos(t),0),
\eeqas
 of the cubic, Bernoulli's lemniscate, and an ellipse, 
 we obtain the  surfaces shown in Figures \ref{cubicfigure},
 \ref{lemfigure} and \ref{ellipsefigure}.

\end{example}

\pagebreak

\providecommand{\bysame}{\leavevmode\hbox to3em{\hrulefill}\thinspace}
\providecommand{\MR}{\relax\ifhmode\unskip\space\fi MR }
\providecommand{\MRhref}[2]{%
  \href{http://www.ams.org/mathscinet-getitem?mr=#1}{#2}
}
\providecommand{\href}[2]{#2}


\begin{thebibliography}{10}


\bibitem{akns1}
M~J Ablowitz, D~J Kaup, A~C Newell, and H~Segur, \emph{Method for solving the
  sine-{G}ordon equation}, Phys. Rev. Lett. \textbf{30} (1973), 1262--1264.
  
\bibitem{galvezetal2007}
J~Aledo, R~Chaves, and J~G\'alvez, \emph{The {C}auchy problem for improper
  affine spheres and the {H}essian one equation}, Trans. Amer. Math. Soc.
  \textbf{359} (2007), 4183--4208.

\bibitem{algami2009}
J~A Aledo, J~A G\'alvez, and P~Mira, \emph{A {d}'Alembert formula for flat
  surfaces in the 3-sphere}, J. Geom. Anal. \textbf{19} (2009), 211--232.

\bibitem{aliaschavesmira2003}
L~J Al\'ias, R~M~B Chaves, and P~Mira, \emph{Bj\"orling problem for maximal
  surfaces in {L}orentz-{M}inkowski space}, Math. Proc. Cambridge Philos. Soc.
  \textbf{134} (2003), 289--316.

\bibitem{bobenko1994}
A~I Bobenko, \emph{Surfaces in terms of 2 by 2 matrices. {O}ld and new
  integrable cases}, Harmonic maps and integrable systems, Aspects Math., no.
  E23, Vieweg, 1994, pp.~83--127.

\bibitem{sbjorling}
D~Brander, \emph{Singularities of spacelike constant mean curvature surfaces in
  {L}orentz-{M}inkowski space}, arxiv:0912.5081.

\bibitem{jgp}
\bysame, \emph{Loop group decompositions in almost split real forms and
  applications to soliton theory and geometry}, J. Geom. Phys. \textbf{58}
  (2008), 1792--1800.

\bibitem{branderdorf}
D~Brander and J~Dorfmeister, \emph{Generalized {DPW} method and an application
  to isometric immersions of space forms}, Math. Z. \textbf{262}, 143--172.

\bibitem{bjorling}
D~Brander and J~F Dorfmeister, \emph{The {B}j\"orling problem for non-minimal
  constant mean curvature surfaces}, Comm. Anal. Geom. \textbf{18} (2010),
  171--194.

\bibitem{dhkw}
U~Dierkes, S~Hildebrandt, A~K\"uster, and O~Wohlrab, \emph{Minimal surfaces.
  {I}. {B}oundary value problems}, Grundlehren der {M}athematischen
  {W}issenschaften, vol. 295, Springer-Verlag, 1992.

\bibitem{dit2000}
J~Dorfmeister, J~Inoguchi, and M~Toda, \emph{Weierstrass-type representation of
  timelike surfaces with constant mean curvature}, Contemp. Math. \textbf{308},
  77--99.

\bibitem{dorfwu2008}
J~F Dorfmeister and H~Wu, \emph{Construction of constant mean curvature n-noids
  from holomorphic potentials}, Math. Z. \textbf{258} (2008), 773--803.

\bibitem{galvezmira2004}
J~G\'alvez and P~Mira, \emph{Dense solutions to the {C}auchy problem for
  minimal surfaces}, Bull. Braz. Math. Soc. (N.S.) \textbf{35} (2004),
  387--394.

\bibitem{gm2005}
\bysame, \emph{The {C}auchy problem for the {L}iouville equation and {B}ryant
  surfaces}, Adv. Math. \textbf{195} (2005), 456--490.

\bibitem{gm2005-2}
\bysame, \emph{Embedded isolated singularities of flat surfaces in hyperbolic
  3-space}, Calc. Var. Partial Differential Equations \textbf{24} (2005),
  239--260.

\bibitem{krs}
M~Kilian, W~Rossman, and N~Schmitt, \emph{Delaunay ends of constant mean
  curvature surfaces}, Compos. Math. \textbf{144} (2008), 186--220.

\bibitem{kimyang2007}
Y~W Kim and S~D Yang, \emph{Prescribing singularities of maximal surfaces via a
  singular {B}j\"orling representation formula}, J. Geom. Phys. \textbf{57}
  (2007), 2167--2177.

\bibitem{krichever}
I~M Krichever, \emph{An analogue of the d'{A}lembert formula for the equations
  of a principal chiral field and the sine-{G}ordon equation}, Dokl. Akad. Nauk
  SSSR \textbf{253} (1980), no.~2, 288--292.

\bibitem{lopez2000}
R~L\'opez, \emph{Timelike surfaces with constant mean curvature in {L}orentz
  three-space}, Tohoku Math. J. \textbf{52} (2000), 515--532.

\bibitem{melkosterling}
M~Melko and I~Sterling, \emph{Application of soliton theory to the construction
  of pseudospherical surfaces in {${\bf R}^3$}}, Ann. Global Anal. Geom.
  \textbf{11} (1993), 65--107.

\bibitem{mira2006}
P~Mira, \emph{Complete minimal {M}\"obius strips in {${\mathbb R}^n$} and the
  {B}j\"orling problem}, J. Geom. Phys. \textbf{56} (2006), 1506--1515.

\bibitem{PreS}
A~Pressley and G~Segal, \emph{Loop groups}, Oxford Mathematical Monographs,
  Clarendon Press, Oxford, 1986.

\bibitem{sym1985}
A~Sym, \emph{Soliton surfaces and their applications}, Geometric aspects of the
  {E}instein equations and integrable systems, Lecture Notes in Physics, vol.
  239, Springer, 1985, pp.~154--231.

\bibitem{terng1997}
C~L Terng, \emph{Solitons and differential geometry}, J. Differential Geometry
  \textbf{45} (1997), 407--445.

\bibitem{todaagag}
M~Toda, \emph{Initial value problems of the sine-{G}ordon equation and geometric
  solutions}, Ann. Global Anal. Geom. \textbf{27} (2005), 257--271.

\end{thebibliography}
\end{document}